\documentclass[reqno,11pt]{amsart}

\usepackage{amsmath,amsfonts,amssymb,graphicx,amsthm,enumerate,url}
\usepackage[noadjust]{cite}
\usepackage{stmaryrd}

\usepackage{xifthen,xcolor,tikz,setspace}
\usetikzlibrary{decorations.pathmorphing,patterns,shapes,calc}
\usepackage{mathrsfs}
\usetikzlibrary{decorations.pathmorphing,patterns,shapes,calc}
\usepackage[noadjust]{cite}
\definecolor{refkey}{gray}{.75}
\definecolor{labelkey}{gray}{.5}

\usepackage[colorlinks=true]{hyperref}
\colorlet{DarkGreen}{green!80!black}
\colorlet{DarkBlue}{blue!50!black}
\colorlet{DarkGray}{gray!60!black}
\colorlet{DarkPurple}{purple!60!black}
\numberwithin{equation}{section}

\renewcommand{\restriction}{\mathord{\upharpoonright}}
\renewcommand{\epsilon}{\varepsilon}
\newcommand{\given}{\;\big|\;}
\newcommand{\one}{\mathbf{1}}

\usepackage{color}
 \definecolor{refkey}{gray}{.5}
 \definecolor{labelkey}{gray}{.5}
\definecolor{light}{gray}{.9}

\newtheorem{maintheorem}{Theorem}

\newtheorem{theorem}{Theorem}[section]
\newtheorem*{theorem*}{Theorem}
\newtheorem{lemma}[theorem]{Lemma}
\newtheorem{claim}[theorem]{Claim}
\newtheorem{proposition}[theorem]{Proposition}

\newtheorem{fact}[theorem]{Fact}
\newtheorem{corollary}[theorem]{Corollary}

\theoremstyle{definition}{

\newtheorem*{definition*}{Definition}

}

\renewcommand{\P}{\mathbb P}

\newcommand{\cA}{\ensuremath{\mathcal A}}

\newcommand{\cC}{\ensuremath{\mathcal C}}

\newcommand{\cG}{\ensuremath{\mathcal G}}

\newcommand{\cL}{\ensuremath{\mathcal L}}

\newcommand{\sL}{{\ensuremath{\mathscr L}}}

\renewcommand{\epsilon}{\varepsilon}
\DeclareMathOperator{\Bin}{Bin}

\newcommand{\tmix}{t_{\textsc{mix}}}

\newcommand{\gap}{\text{\tt{gap}}}
\newcommand{\tv}{{\textsc{tv}}}

\newcommand{\rc}{{\textsc{rc}}}

\title[Exponentially slow mixing in mean-field Swendsen--Wang dynamics]{Exponentially slow mixing in the \\ mean-field Swendsen--Wang dynamics}
\author{Reza Gheissari}
\address{R.\ Gheissari\hfill\break
Courant Institute\\ New York University\\
251 Mercer Street\\ New York, NY 10012, USA.}
\email{reza@cims.nyu.edu}

\author{Eyal Lubetzky}
\address{E.\ Lubetzky\hfill\break
Courant Institute\\ New York University\\
251 Mercer Street\\ New York, NY 10012, USA.}
\email{eyal@courant.nyu.edu}

\author{Yuval Peres}
\address{Y.\ Peres\hfill\break
Microsoft Research\\
1 Microsoft Way\\ Redmond, WA 98052, USA.}
\email{peres@microsoft.com}

\begin{document}

\begin{abstract}
Swendsen--Wang dynamics for the Potts model was proposed in the late 1980's as an alternative to single-site heat-bath dynamics, in which global updates allow this MCMC sampler to switch between metastable states and ideally mix faster. 
Gore and Jerrum (1999) found that this dynamics may in fact exhibit slow mixing: they showed  that, for the Potts model with $q\geq 3$ colors on the complete graph on $n$ vertices at the critical point $\beta_c(q)$, Swendsen--Wang dynamics has  $\tmix\geq \exp(c\sqrt n)$. 
The same lower bound was extended to the critical window $(\beta_s,\beta_S)$ around $\beta_c$ by Galanis \emph{et al.\ }(2015), as well as to the corresponding mean-field FK model by Blanca and Sinclair~(2015).
In both cases, an upper bound of $\tmix \leq \exp(c' n)$ was known.
Here we show that the mixing time is truly exponential in $n$: namely, $\tmix \geq \exp (cn)$ for  Swendsen--Wang dynamics when $q\geq 3$ and $\beta\in(\beta_s,\beta_S)$, and the same bound holds for the related MCMC samplers for the mean-field FK model when $q>2$.
\end{abstract}

%{\mbox{}
%\vspace{-1cm}
\maketitle
%\vspace{-1cm}
%}

\section{Introduction}
The mean-field $q$-state Potts model is a canonical statistical physics model extending the Curie--Weiss Ising model ($q=2$) to $q\in\mathbb N$ possible states; for $q\geq 3$, it is one of the simplest models to exhibit a discontinuous (first-order) phase transition. Formally, the mean-field $q$-state Potts model with parameter $\beta$ is a probability distribution $\mu_{n,\beta,q}$ over $\{1,\ldots,q\}^n$, given by
$\mu_{n,\beta,q}(\sigma)\propto \exp(\frac\beta n  H(\sigma))$, where $H(\sigma)=\sum_{i<j}\one{\{\sigma_i=\sigma_j\}}$. The model exhibits a phase transition at $\beta=\beta_c(q)$ from a disordered phase ($\beta<\beta_c$), where the sizes of all $q$ color classes concentrate around $n/q$, to an ordered phase ($\beta>\beta_c$), where there is typically one color class of size $a_\beta n$ for $a_\beta>1/q$ (see~\S\ref{sec:prelims}).

As a means of overcoming low-temperature bottlenecks in the energy landscape (dominant color classes), Swendsen and Wang~\cite{SW87} introduced a non-local reversible Markov chain, relying on the random cluster (FK) representation of the Potts model.
The mean-field FK model is the generalization of $\cG(n,p)$---the Erd\H{o}s--R\'enyi random graph---parametrized by $(p=\frac{\lambda}n,q)$, in which the probability of a graph $G=(V,E)$, identified with $\omega\in\Omega_\rc :=\{0,1\}^{\binom{n}2}$, is given by  $\pi_{n,\lambda,q} (\omega)\propto p^{|E|}(1-p)^{\binom{n}2-|E|} q^{k(G)}$, where $k(G)$ is the number of connected components of $G$ (clusters of $\omega$). 

Via the Edwards--Sokal coupling~\cite{EdSo88} of the $q$-state Potts model at inverse temperature $\beta/n$ and the FK model with parameters $(p,q)$ with $p=1-e^{-\beta/n}$, the mean-field Swendsen--Wang dynamics can be formulated as follows: consider a mean-field Potts configuration $\sigma$ with $V_1,...,V_q$ being the sets of vertices $V_i=\{x:\sigma_x=i\}$. An update of the dynamics, started from $\sigma$, first samples,  independently for every $i=1,...,q$, a configuration $G_i\sim \cG(|V_i|,p)$ on the subgraph of $V_i$, forming  an FK configuration $\omega$ as the union of the $G_i$'s; then, it assigns an i.i.d.\ color $X_C\sim \mbox{Uni}(\{1,...,q\})$ to each cluster $C$ in $\omega$, and for every $x\in C$, sets $\sigma'_x=X_c$ in the new state $\sigma'$ of the Markov chain.
\begin{figure}

\begin{tikzpicture}
    \node[scale=1.2] (plot1) at (0,0) {};

	    \node (fig1) at (1.6,-1.1) {
	\includegraphics[width=75pt]{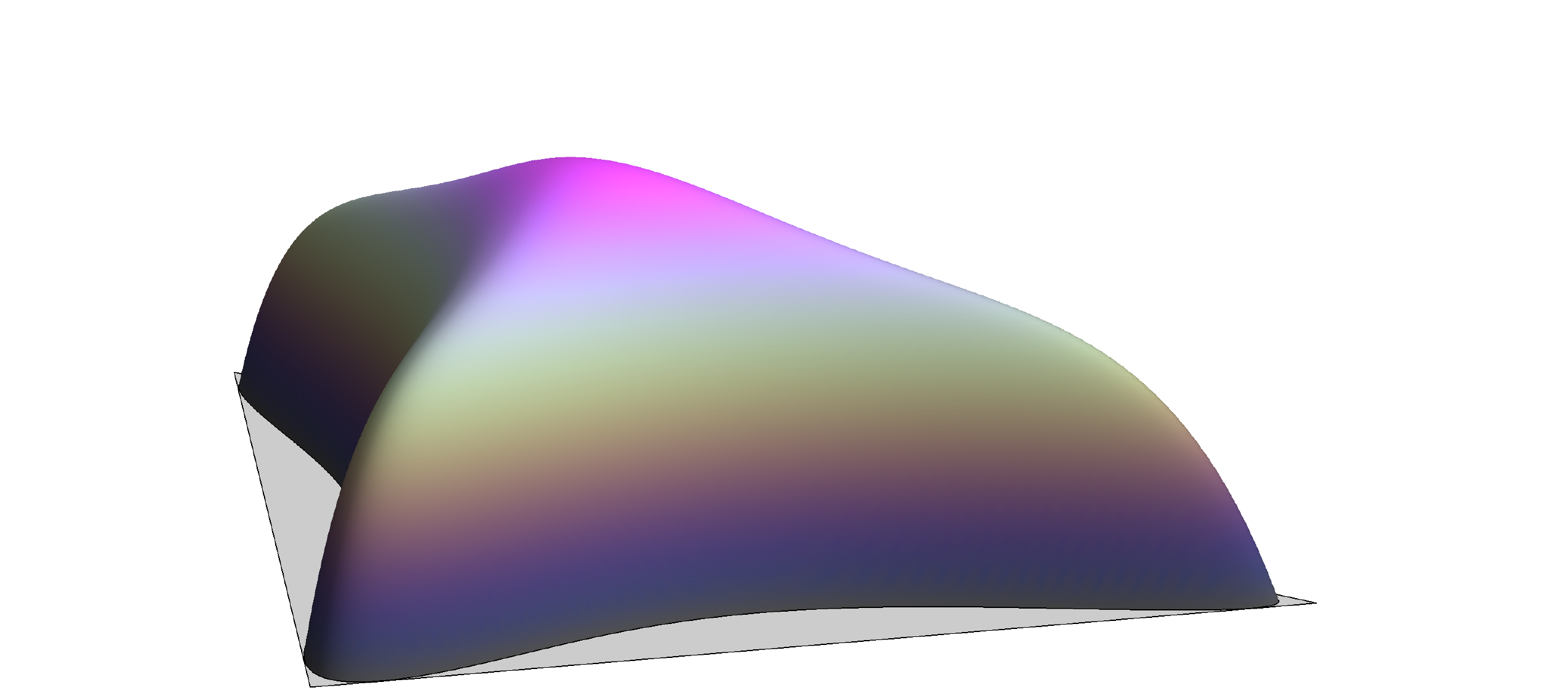}};
	    \node (fig2) at (5,-1.15) {
	\includegraphics[width=75pt]{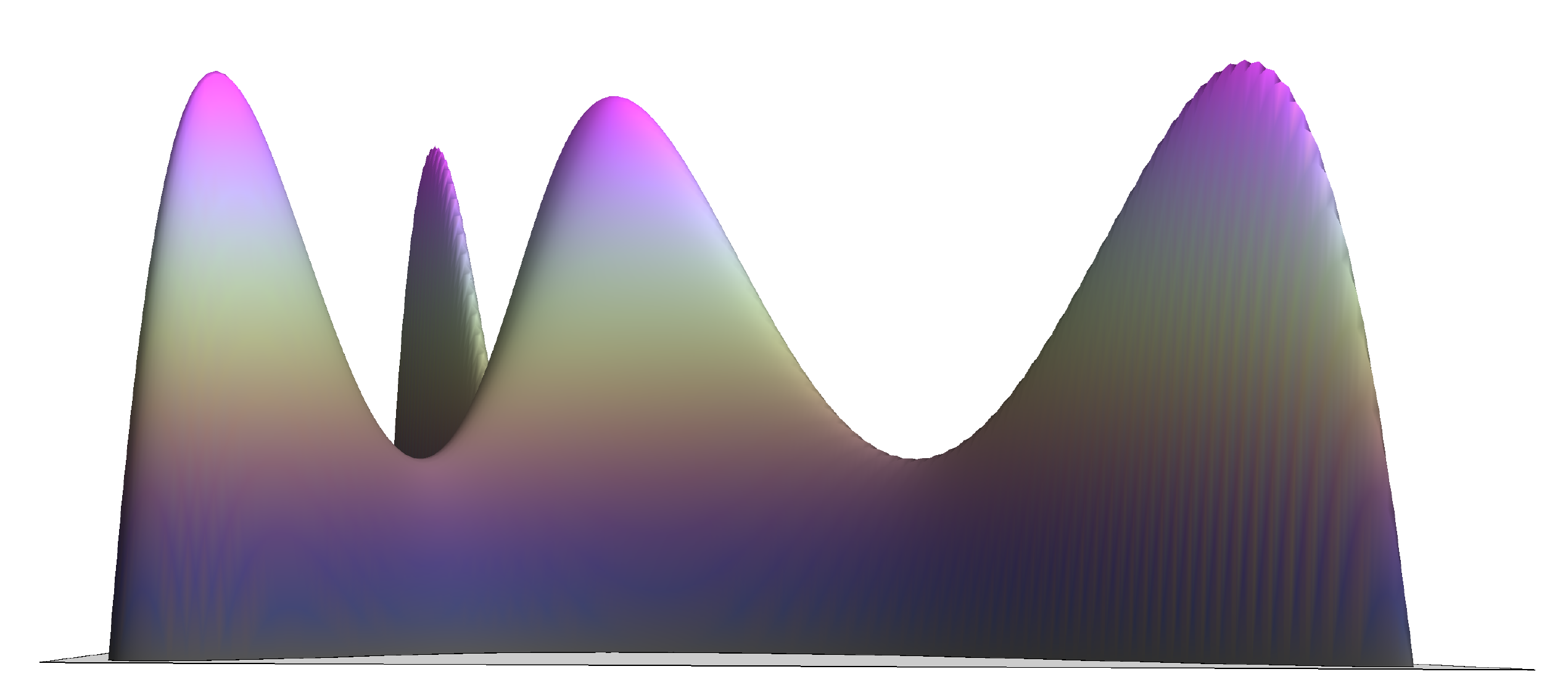}};
	    \node (fig3) at (8.1,-1.15) {
	\includegraphics[width=75pt]{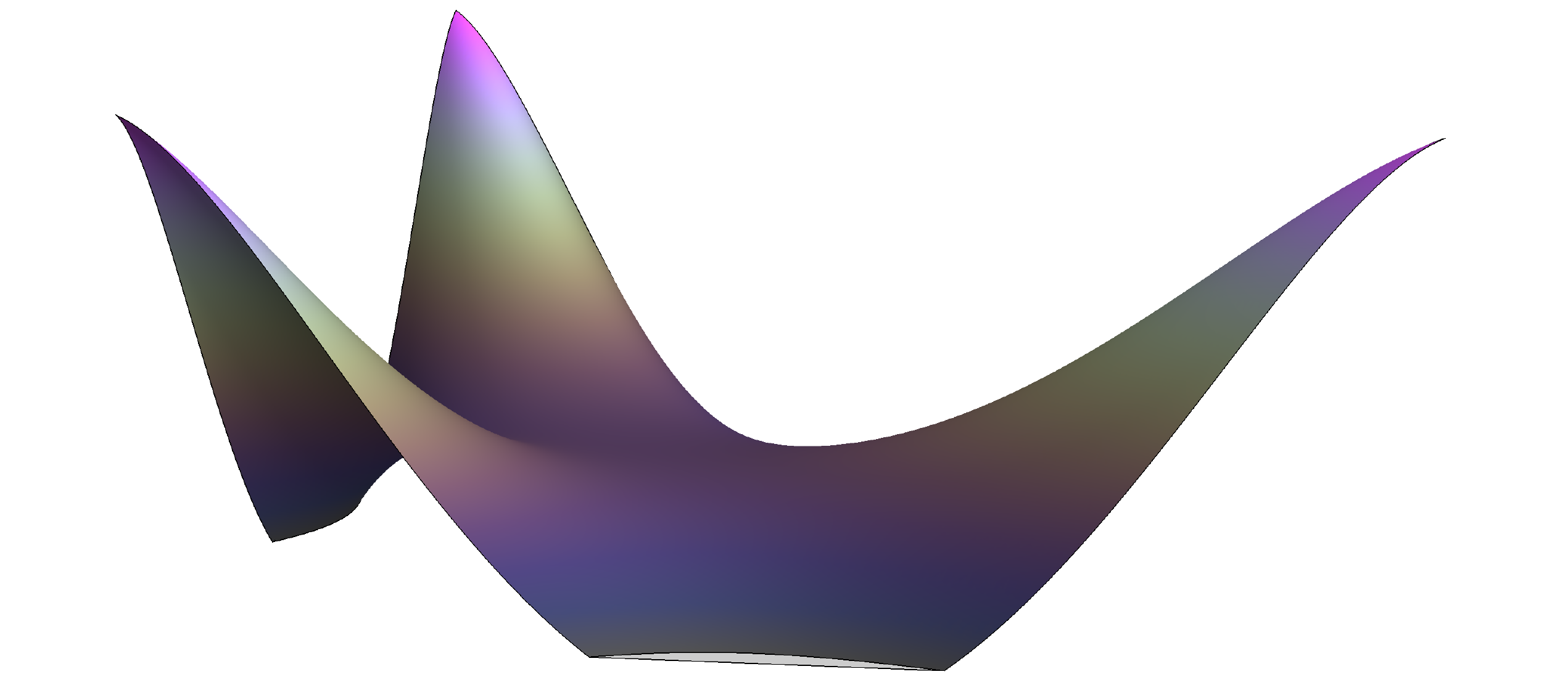}};

    \begin{scope}[shift={(plot1.south west)},x={(plot1.south east)},y={(plot1.north west)}, font=\small]

     \draw[draw=black,thick, ->] (0,0)--(30,0);
     \draw[draw=black,thick, ->] (0,0)--(0,15.5);
     
     \draw[draw=DarkBlue, ultra thick] (0,1)--(9.75,1);

     \draw[draw=DarkBlue, ultra thick] (20,2.5)--(30,2.5);
     \draw[draw=DarkBlue, ultra thick] (10.25,13.5)--(19.75,13.5);
     \draw[draw=DarkBlue,dashed] (10.25,8.5)--(19.75,8.5);

     \draw[color=DarkBlue]  (10,1) circle (0.25) node[above right] {};
     \draw[color=DarkBlue, fill=DarkBlue]  (10,4) circle (0.25) node[above right] {};
     \draw[color=DarkGreen, pattern=north west lines, pattern color=DarkGreen, thick]  (10,4) circle (0.25) node[above right] {};

     \draw[color=DarkBlue]  (10,8.5) circle (0.25) node[above right] {};
     \draw[color=DarkBlue]  (10,13.5) circle (0.25) node[above right] {};
     \draw[color=DarkGreen]  (10,13.65) circle (0.25) node[above right] {};

     \draw[color=DarkBlue]  (20,8.5) circle (0.25) node[above right] {};
     \draw[color=DarkBlue]  (20,13.5) circle (0.25) node[above right] {};
     \draw[color=DarkBlue, fill=DarkBlue]  (20,2.5) circle (0.25) node[above right] {};
     
               \draw[draw=DarkGreen, ultra thick] (10.25,13.65)--(30,13.65);
     \draw[draw=DarkGreen, ultra thick] (0,2.5)--(9.75,2.5);
     \draw[color=DarkGreen]  (10,2.5) circle (0.25) node[above right] {};
     
     \draw[color=black, thick, ->] (14.7,9.5)--(14.7,12.5);
     \draw[color=black, thick, ->] (15.3,9.5)--(15.3,12.5);

     \node[font=\small] at (10,-1.1) {$\beta_s$};
     \node[font=\small] at (15,-1.1) {$\beta_c$};
     \node[font=\small] at (20,-1.1) {$\beta_S$};
     \node[font=\tiny] at (10,0) {$|$};
     \node[font=\tiny] at (15,0) {$|$};
     \node[font=\tiny] at (20,0) {$|$};

     \node[font=\small] at (-1.4,1) {$ c_\beta$};
     \node[font=\small] at (-2.65,2.6) {$ c_\beta \log n$};
     \node[font=\small] at (-2,4.5) {$ cn^{1/3}$};
     \node[font=\small] at (-2.2,8.8) {$e^{c_\beta \sqrt n}$};
     \node[font=\small] at (-1.6,13.6) {$e^{c_\beta n}$};
     \node at (-1.5,15.5) {$\tmix$};
     
\end{scope}

\end{tikzpicture}
\vspace{-.15cm}
\begin{caption} {The mixing times of mean-field Potts Glauber (green) and Swendsen--Wang (red) dynamics when $q>2$ as $\beta$ varies; the dashed line represents the previous lower bound~\cite{GoJe99,GSV15} when $\beta\in (\beta_s,\beta_S)$}
\end{caption}
\end{figure}

As apparent from the second (coloring) stage of the Swendsen--Wang algorithm, it can seamlessly jump between the $q$ ordered low-temperature metastable states where one color is dominant. 
It was thus expected that this MCMC sampler would converge quickly to equilibrium at all temperatures; e.g., its total variation mixing time $\tmix$, formally defined in~\S\ref{sec:prelims}, would be at most polynomial in the system size for all $\beta>0$.
 
Indeed, at $q=2$ (the Ising model) Cooper, Dyer, Frieze and Rue~\cite{CDFR00} proved that, on the complete graph, Swendsen--Wang has  $\tmix = O(\sqrt{n})$ at all $\beta$ (it was later shown in~\cite{LNNP14} that  $\tmix \asymp n^{1/4}$ at $\beta_c$ while  $\tmix=O(\log n)$ at $\beta\neq\beta_c$), and Guo and Jerrum~\cite{GuJe16} recently showed that for \emph{any} $n$-vertex graph and all $\beta$, Swendsen--Wang has $\tmix = n^{O(1)}$
(this is in contrast to single-site dynamics, where $\tmix \geq \exp(cn)$ at low temperature~\cite{CDL12}).

Countering this intuition, however, Gore and Jerrum~\cite{GoJe99} found in 1999 that, for any $q\geq 3$, the Swendsen--Wang dynamics for the mean-field $q$-state Potts model has $\tmix \geq \exp(c\sqrt{n})$ for some $c(q)>0$ at its critical point $\beta_c(q)$. This is a consequence of the discontinuity of the phase transition of the mean-field Potts model for $q\geq 3$, where at $\beta_c(q)$, both the $q$ ordered phases (with one dominant color class) and the disordered phase (with all color classes having roughly $n/q$ sites) are metastable. 

On the lattice $(\mathbb Z/n\mathbb Z)^d$, the Potts model exhibits a discontinuous phase transition for some choices of $q$ (depending on $d$); there it was shown in~\cite{BCT12}, following~\cite{BCFKTVV99},
that Swendsen--Wang dynamics in fact has $\tmix \geq \exp (cn^{d-1})$ for all $q$ sufficiently large, suggesting that an exponential lower bound in $n$ should also hold in mean-field, believed to approximate high-dimensional tori. (The matching upper bound of~\cite{BCT12} applies to general graphs and translates to  $\tmix \leq \exp(c'n)$ on the complete graph.) On $\mathbb Z^2$, this lower bound was extended~\cite{GL16a} to $q$ where the phase transition is first-order (all $q>4$).

For the Glauber dynamics of the mean-field Potts model, when $q\geq 3$, the mixing time for all $\beta$ was characterized in~\cite{CDL12}, where it was shown that, in discrete-time, $\tmix$ has order $n\log n$ at $\beta<\beta_s$, order $ n^{4/3}$ at $\beta=\beta_s$, and finally $\tmix\geq \exp(c n)$ at $\beta>\beta_s$, where $\beta_s$ is the spinodal point corresponding to the onset of $q$ ordered metastable phases. Recently, Galanis, \v{S}tefankovic and Vigoda~\cite{GSV15} 
analyzed the  mixing time of the analogous mean-field Swendsen--Wang dynamics, finding it to mix in polynomial time\footnote{It was shown in that work that $\tmix = O(\log n)$ for $\beta\notin [\beta_s,\beta_S)$, whereas $\tmix\asymp n^{1/3}$ at $\beta=\beta_s$.}  both at high temperature and---unlike Glauber dynamics---at low temperatures, for all $\beta$ outside a critical window $(\beta_s,\beta_S)$ around $\beta_c$, where the critical point $\beta_S$ (mirroring the spinodal point $\beta_s$) marks the disappearance of metastability of the disordered phase. 

For $\beta\in(\beta_s,\beta_S)$, Swendsen--Wang was shown in~\cite{GSV15} to  slow down to
 $\tmix \gtrsim \exp({c\sqrt n})$ (extending the lower bound at $\beta=\beta_c$ due to Gore and Jerrum).
Analogously, for the  related Glauber dynamics for the mean-field FK model (see~\S\ref{sec:prelims} for precise definitions) with $q>2$, Blanca and Sinclair~\cite{BlSi14} proved that $\tmix\geq \exp({c\sqrt n})$
 whenever $\lambda=np$ is in the critical window $ (\lambda_s,\lambda_S)$.
 The fact that three significant papers, over a period of almost twenty years, all presented a lower bound of the form $\exp(c\sqrt{n})$, left open the possibility that this is the true order of the mixing time inside the critical window.

Our main result is  that 
the mixing time of the mean-field Swendsen--Wang dynamics  is truly exponential in $n$ at criticality, similar to the single-site Glauber dynamics.

\begin{figure}
\vspace{-0.5cm}
\centering
  \begin{tikzpicture}
    \node (fig1) at (0,0) {
	\includegraphics[width=.33\textwidth,trim={0 30pt 0 0},clip]{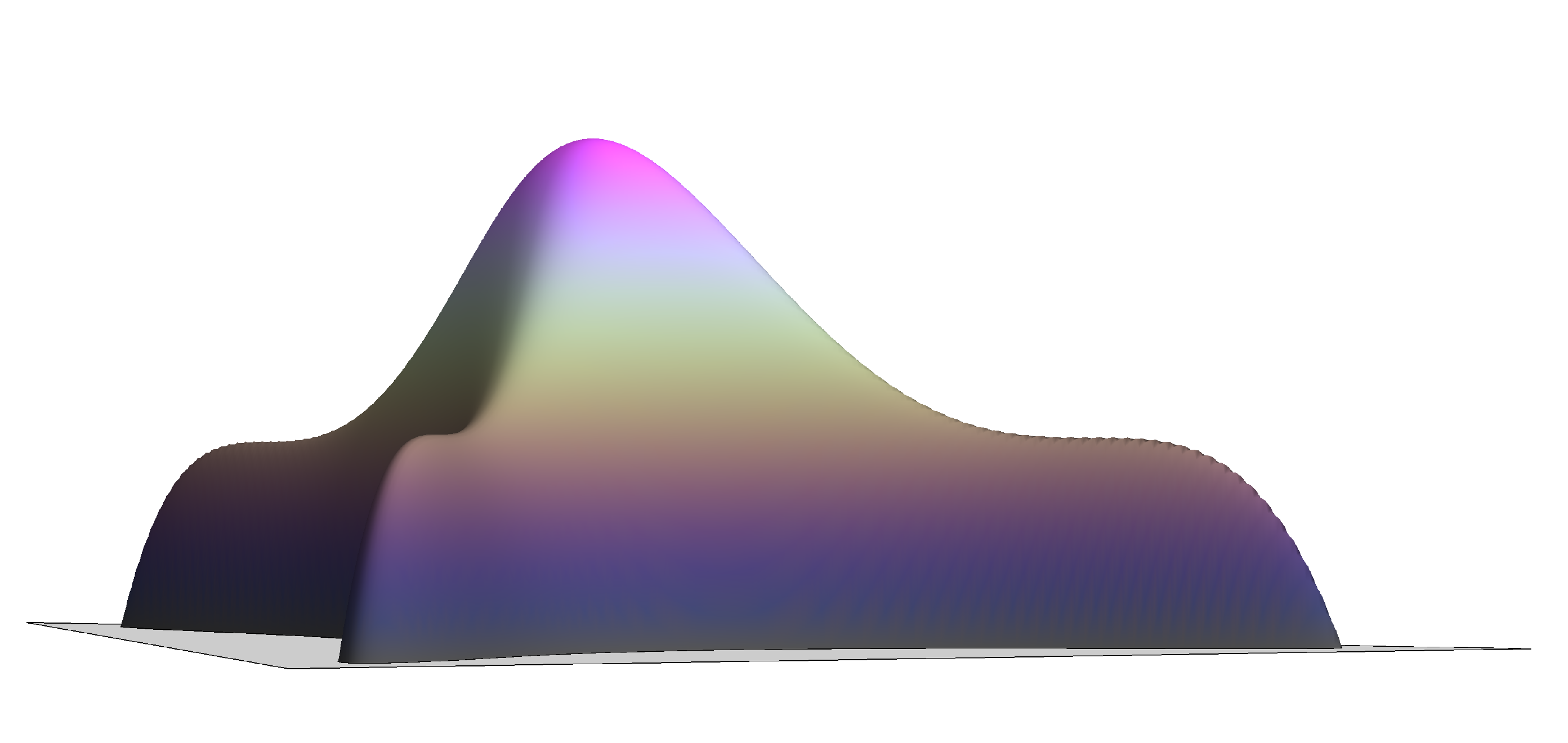}
	\includegraphics[width=.33\textwidth]{potts_mf_bc_2773_noax}
	\includegraphics[width=.33\textwidth,trim={0 10pt 0 0},clip]{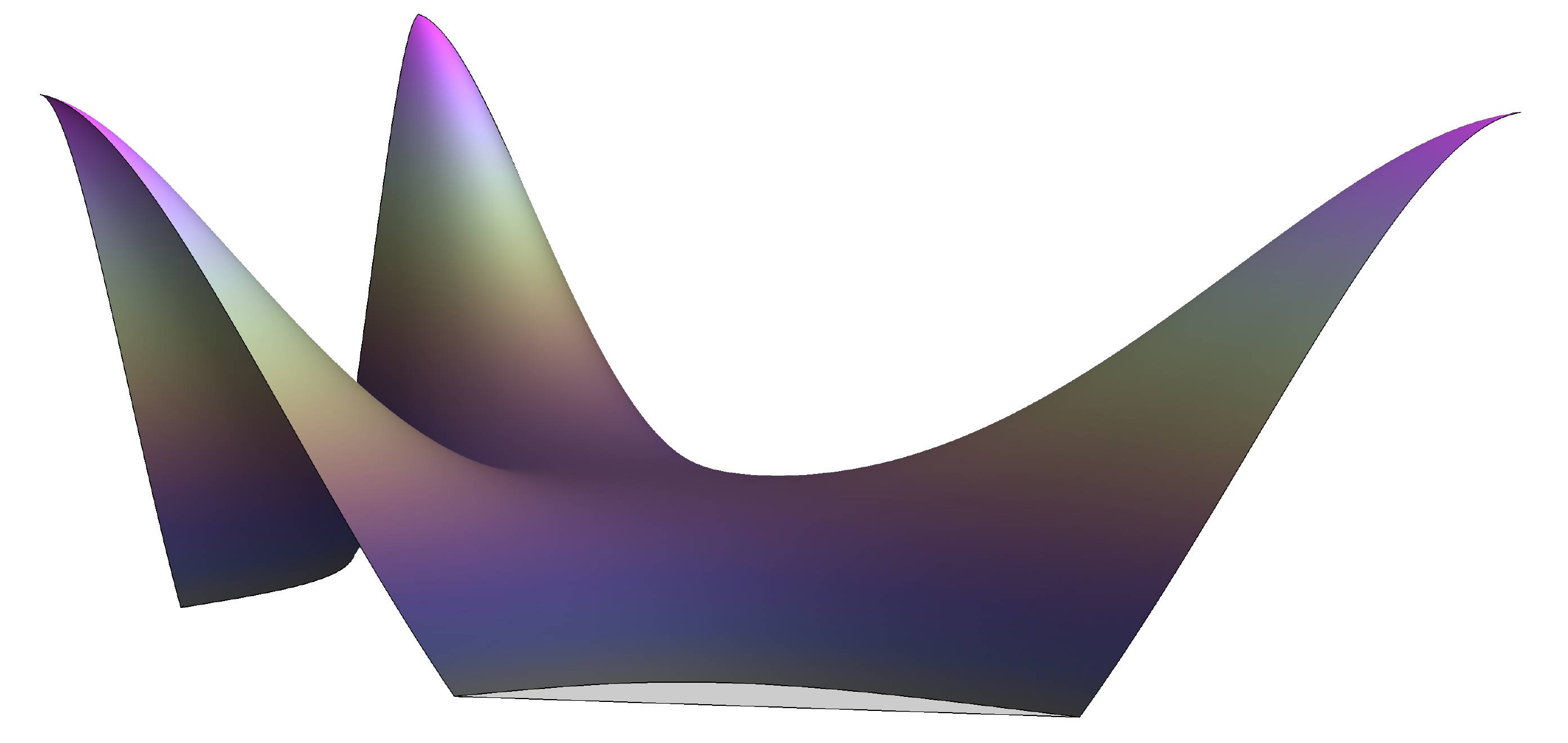}
	};
    \begin{scope}[shift={(fig1.south west)},x={(fig1.south east)},y={(fig1.north west)}]
      \node[font=\small] at (0.15,-0.1) {$\beta_s\approx 2.745$};
      \node[font=\small] at (0.5,-0.1) {$\beta_c\approx 2.773$};
      \node[font=\small] at (0.82,-0.1) {$\beta_S=3$};
      \end{scope}
    \node (fig2) at (-2.75,-2.8) {
    \includegraphics[width=.33\textwidth]{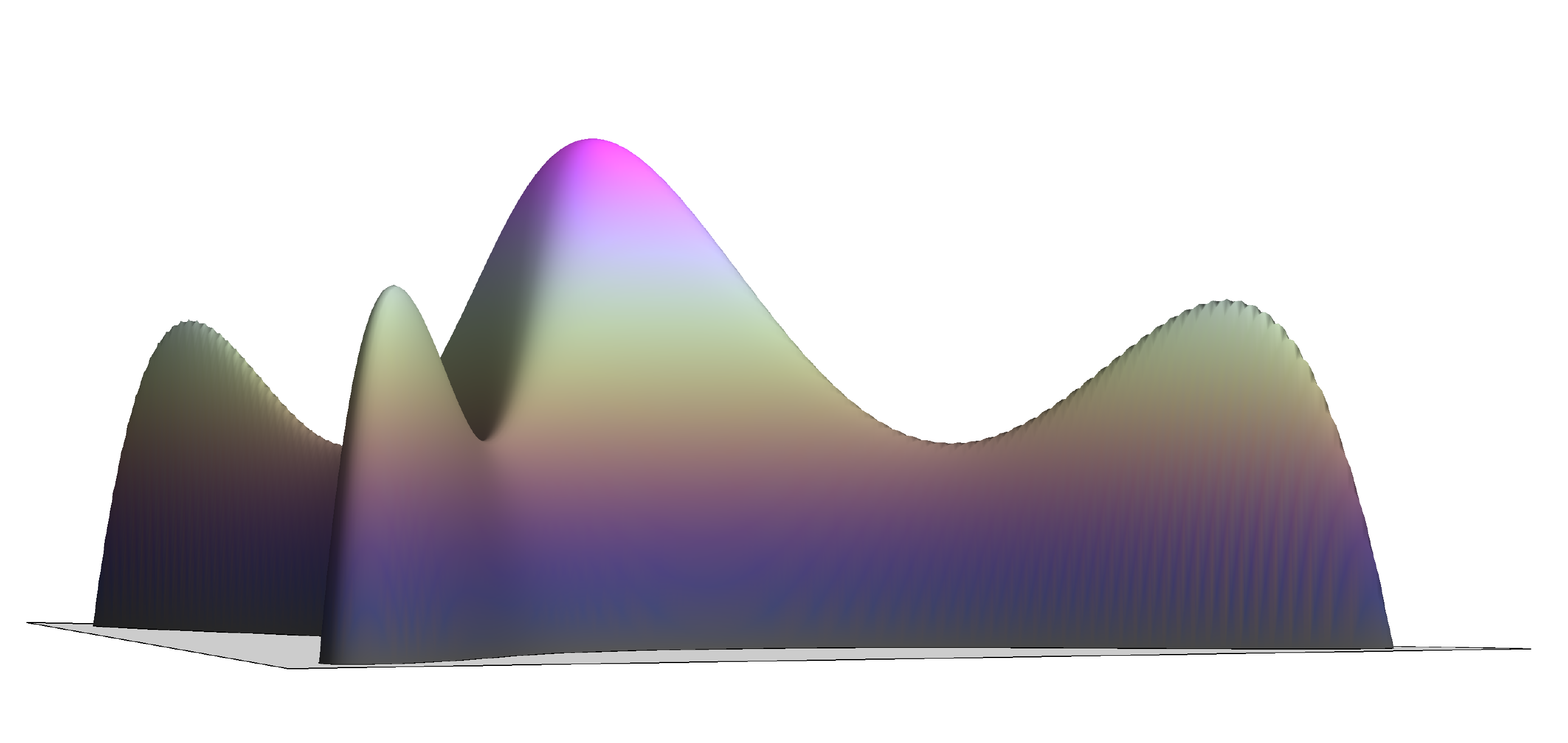}
    };
    \node (fig3) at (2.75,-2.8) {
    \includegraphics[width=.33\textwidth]{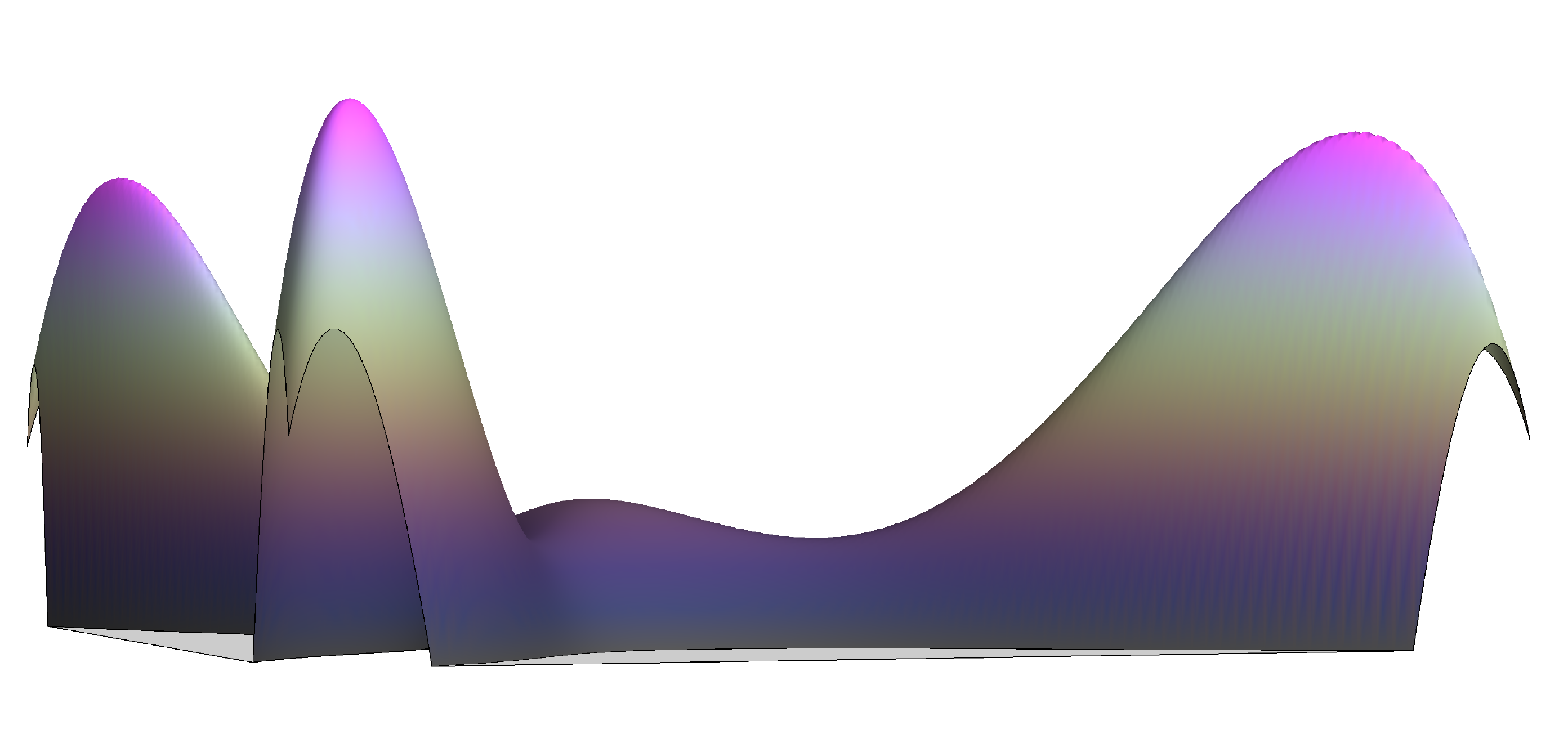}
    };
    \begin{scope}[shift={(fig2.south west)},x={(fig2.south east)},y={(fig2.north west)}]
      \node[font=\small] at (0.5,0) {$\beta_s<\beta<\beta_c$};
      \node[font=\small] at (1.55,0) {$\beta_c<\beta<\beta_S$};
      \end{scope}
  \end{tikzpicture}
\vspace{-0.9cm}
 \caption{The free energy landscape of the $3$-state Potts model in the metastability window $\beta_s\leq \beta\leq \beta_S$. The three outer peaks correspond to the ordered phases;  middle peak corresponds to the disordered phase. }
  \label{fig:potts_sw}
\vspace{-0.2cm}
\end{figure}

\begin{maintheorem}\label{mainthm:sw}
Let $q\geq 3$ be a fixed integer, and consider the Swendsen--Wang dynamics for the $q$-state mean-field Potts model on $n$ vertices at inverse temperature $\beta \in (\beta_s,\beta_S)$. There exists some $c(\beta,q)>0$ such that, for all $n$ large enough, $\tmix \geq \exp (cn)$.
\end{maintheorem}

The case of non-integer $q$ (the mean-field FK model) is more delicate: the analogue of Swendsen--Wang  in this setting is Chayes--Machta dynamics~\cite{ChMa97}, which we analyze via a recursive application of the fundamental lemma of Bollob\'as, Grimmett and Janson~\cite{BGJ}. 
As in~\cite{BlSi14}, comparison results of~\cite{Ul13} extend the result to heat-bath Glauber dynamics.

\begin{maintheorem}\label{mainthm:fk}
Fix $q>2$, and consider  Glauber dynamics for the mean-field FK model on $n$ vertices with parameters $(p=\frac \lambda n,q)$ where $\lambda\in (\lambda_s,\lambda_S)$. There exists $c(p,q)>0$ such that $\tmix \geq \exp(cn)$ for large enough $n$. The same holds for  Chayes--Machta dynamics.
\end{maintheorem}

To outline our approach for proving Theorems~\ref{mainthm:sw}--\ref{mainthm:fk}, we first sketch the argument of~\cite{GoJe99},  thereafter adapted to $\beta\in[\beta_s,\beta_S)$ in~\cite{GSV15} and to the FK model in~\cite{BlSi14}. Starting from a Potts configuration where each color class has $\frac nq \pm\epsilon n$ vertices, since $\beta<\beta_S$, for small enough $\epsilon$, this corresponds to a subcritical Erd\H{o}s-R\'enyi random graph $\cG(n,p)$ in the first stage of the Swendsen--Wang dynamics. The exponential tail of component sizes in this regime shows that, for a sequence $k=k(n)$, with probability at least $1-n\exp(-c k)$, no cluster in the edge configuration we obtain is larger than $k$; on this event, the component sizes $\sL_i$ satisfy $\sum_i \sL_i^2 \leq k \sum \sL_i = n k$, thus 
by Hoeffding's inequality, with probability $1-O(\exp[-\epsilon^2 n/ (2k)])$, every new color class will have $n/q \pm\epsilon n$ vertices, and in particular no dominant color class would emerge.
In this argument, choosing $k \asymp \sqrt n$ balances the two probability estimates to $1-\exp(-c\sqrt{n})$. However, at $\beta\geq\beta_c$, the Potts model does admit a dominant color class with positive (uniformly bounded away from 0) probability, thus the mixing time is at least $ \exp(c\sqrt n)$. 

In order to  improve this lower bound into $\exp( c n)$ per Theorem~\ref{mainthm:sw}, instead of looking at the size of the largest component after the $\cG(n,p)$ stage of the dynamics, we consider $S_M$, the set of vertices in connected components of size larger than $M$. We show that, whenever the $\cG(n,p)$ stage is subcritical and $M$ is sufficiently large, the probability that $|S_M| > \rho n$ is at most $\exp(-c\rho n)$. Moreover, given $|S_M|\leq \rho n$, Hoeffding's inequality implies that, following the second stage of the dynamics, all the new color classes will have $n/q \pm \epsilon n$ vertices except with probability $\exp[-2 (\frac{\epsilon-\rho}M)^2 n]$, yielding $\tmix \geq \exp(cn)$.  The proof of Theorem~\ref{mainthm:fk} follows a similar path, yet involves additional equilibrium estimates on the conditional probabilities under $\pi_{n,\lambda,q}$, as the Chayes--Machta dynamics resamples a strict subset of the configuration in each step.

\section{Preliminaries}\label{sec:prelims}

Throughout this paper, we use the notation $ f \lesssim g$ for two sequences $f(n),g(n)$ to denote $f= O(g)$, and let $f \asymp g$ denote $f \lesssim g \lesssim f$.
We re-parametrize the FK and Potts models by $\lambda$ instead of $p$ and $\beta$ via the relations $p=\lambda/n$ and $\lambda/n=1-e^{-\beta/n}$, to allow us to treat the FK and Potts models in a unified manner. We will consider these models  on the complete graph on $n$ vertices, $G=(V,E)= (\{1,...,n\},\{ij\}_{1\leq i< j\leq n})$.

Denote by $\mu_{n,\lambda,q}$, the Potts measure (with $\beta$ such that $\lambda/n=1-e^{-\beta/n}$) and by $\pi_{n,\lambda,q}$ the corresponding FK measure with $p=\lambda/ n$ on the complete graph on $n$ vertices. The FK model with $q=1$ corresponds precisely to the Erd\H{o}s--R\'enyi random graph $\cG(n,p)$ and we use the shortened notation $\pi_{n,\lambda}= \pi_{n,\lambda,1}$. We occasionally use $\cG(n,p,q)$ to denote the mean-field FK model given by $\pi_{n,\lambda,q}$.

For any FK configuration $\omega\in\{0,1\}^E$, enumerate the clusters of $\omega$ in decreasing size $\cC_1,\cC_2,...$ and let $\sL_i=|\cC_i|$. For a vertex $x$ let, $\cC_{x}$ denote the cluster to which $x$ belongs.

For all $q\leq  2$ define the critical points $\lambda_s=\lambda_c=\lambda_S=q$ and for $q>2$, define 
\[\lambda_s=\min_{z\geq 0} \left\{z+\frac{qz}{e^z-1}\right\}\,, \qquad \lambda_c= \frac{2(q-1)\log (q-1)}{q-2}\,, \qquad \lambda_S = q\,,
\]
so that for $q>2$, we have $\lambda_s<\lambda_c<\lambda_S$ (see e.g.,~\cite{GSV15,BlSi14}). The critical points $\lambda_s,\lambda_S$ correspond to the parameters of emergence and disappearance of metastability, where at $\lambda=\lambda_c$, the ordered and disordered metastable states have the same free energy. These two critical points can also have the following alternative interpreation~\cite{GSV15}: $\lambda_s$ corresponds to the first uniqueness/non-uniqueness threshold of the $\Delta$-regular infinite tree, and $\lambda_S$ should correspond to a second uniqueness/non-uniqueness threshold of the $\Delta$-regular tree with periodic boundary conditions.

\subsubsection*{The FK and Potts phase transitions} The following give a description of the static phase transition undergone by the mean-field FK and Potts models respectively. 
Let $\Theta_r=\Theta_r(\lambda,q)$ be the largest solution of $e^{-\lambda x}= 1-\frac{qx}{1+(q-1)x}$ so $\Theta_r=\frac {q-2}{q-1}$ when $\lambda=\lambda_c$.
\begin{proposition}[{\cite[Thms.~2.1--2.2]{BGJ},\cite[Thm. 19]{LL06}}]\label{prop:fk-phase-transition}
Consider the $n$-vertex mean-field FK model with parameters $(p,q)$ with $p=\lambda/n$; if $\lambda<\lambda_c(q)$, for every $\epsilon>0$, we have $\lim_{n\to\infty} \pi_{n,\lambda,q}(\sL_1\leq \epsilon n) =1$ whereas if $\lambda>\lambda_c(q)$, for every $\epsilon>0$, we have $\lim_{n\to\infty} \pi_{n,\lambda,q}(\sL_1 \geq (\Theta_r-\epsilon)n) =1$. If $\lambda=\lambda_c(q)$, there exists $\gamma(q)\in (0,1)$ so that for all $\epsilon>0$, $\lim_{n\to\infty}\pi_{n,\lambda,q}(\sL_1 \leq \epsilon n) \geq \gamma$ and $\lim_{n\to\infty}\pi_{n,\lambda,q}(\sL_1\geq (\Theta_r-\epsilon)n)\geq 1-\gamma$.
\end{proposition}

\begin{corollary}\label{cor:Potts-phase-transition}
Consider the mean-field Potts model parametrized by $\lambda=n(1-e^{-\beta/n})$ and $q$. If $\lambda<\lambda_c(q)$, for any $\epsilon>0$, 
\[\lim_{n\to\infty} \mu_{n,\lambda,q} \biggl(\sigma:\max_{r=1,..,q} \Bigl|\tfrac1 n \sum_{i\leq n} \boldsymbol 1\{\sigma_i=r\}-\tfrac 1q\Bigr|<\epsilon \biggr) =1\,,
\]
and if $\lambda>\lambda_c(q)$, then there exists $a(\lambda,q)>q^{-1}$ such that for sufficiently small $\epsilon>0$, 
\[\lim_{n\to\infty} \mu_{n,\lambda,q} \biggl(\sigma:\max_{r=2,...,q} \Bigl\{\Bigl|\tfrac1n \sum _{i\leq n} \boldsymbol 1\{\sigma_i=1\}-a \Bigr|,\Bigl|\tfrac1n \sum_{i\leq n} \boldsymbol1\{\sigma_i=r\}-\tfrac {1-a}{q-1}  \Bigr|\Bigr\}<\epsilon \biggr) =\frac 1q\,.
\]
If $q>2$ and $\lambda=\lambda_c(q)$, there exists $\gamma(q)\in(0,1)$ so that for all sufficiently small $\epsilon>0$,
\begin{align*}
\lim_{n\to\infty} \mu_{n,\lambda,q} & \biggl(\max_{r= 1,..,q} \Bigl|\tfrac1n \sum_{i\leq n} \boldsymbol 1\{\sigma_i=r\}-\tfrac 1q\Bigr|<\epsilon \biggr) \geq \gamma\,, \qquad \mbox{and} \\
\lim_{n\to\infty} \mu_{n,\lambda,q} & \biggl(\max_{r=2,...,q} \Bigl\{\Bigl| \tfrac1n \sum _{i\leq n} \boldsymbol 1\{\sigma_i=1\}-a \Bigr|,\Bigl|\tfrac1n \sum_{i\leq n} \boldsymbol1\{\sigma_i=r\}-\tfrac{1-a}{q-1}\Bigr|\Bigr\}<\epsilon \biggr) \geq \frac{1-\gamma}q\,.
\end{align*}
\end{corollary}

\subsubsection*{Cluster dynamics}
 Swendsen--Wang dynamics for the $q$-state Potts model on $G=(V,E)$ with parameter $\beta$ such that $p=1-e^{-\beta/n}$ is the following discrete-time reversible Markov chain. From a Potts configuration $\sigma$ on $G$, generate a new state $\sigma'$ as follows.
\begin{enumerate}
\item Introduce auxiliary edge variables and for $e=xy\in E$ set $\omega(e)=0$ if $\sigma_x\neq \sigma_y$ on each of the $q$ sets of vertices of $\sigma$ of the same color, $V_1,..,V_q$, independently sample $\omega\restriction_{\{xy:x,y\in V_i\}} \sim \cG(|V_i|,p)$.
\item For every connected component of the resulting $\omega$,
reassign the cluster, collectively, an i.i.d.\ color in $1,...,q$, to obtain the new configuration $\sigma'$.
\end{enumerate}
 Chayes--Machta dynamics for the FK model on $G=(V,E)$ with parameters $(p,q)$, for $q \geq 1$ and $p=\lambda/n$,
is the following discrete-time reversible Markov chain:
 From an FK configuration $\omega\in\Omega_{\rc}$ on $G$,  generate a new state $\omega'\in\Omega_{\rc}$ as follows.
\begin{enumerate}
\item Assign each cluster $C$ of $\omega$ an auxiliary i.i.d.\ variable $X_C\sim\mathrm{Bernoulli}(1/q)$.
\item Resample every $e=xy$ such that $x$ and $y$ belong to \emph{active} clusters ($X_c=1$)
via i.i.d.\ random variables $X_e\sim\mathrm{Bernoulli}(\lambda/n)$, yielding a new configuration $\omega'$.
 \end{enumerate}

Variants of Chayes--Machta dynamics with $1\le k\le \lfloor q\rfloor$ ``active colors" have also been studied, with numerical evidence for  $k=\lfloor q \rfloor$ being the most efficient choice; see~\cite{GOPS11}.

\subsubsection*{Glauber dynamics for the FK model} Swendsen--Wang dynamics is closely related to the FK model; much of the analysis of Swendsen--Wang dynamics on general graphs has been via the Glauber dynamics for the corresponding FK model. Discrete-time Glauber dynamics~\cite{Gl63} for the FK model on $G=(V,E)$  with $p=\lambda/n$ is  as follows: select an edge  $e=xy$ in $E$ uniformly at random and update $\omega(e)$ according to $\pi_{n,\lambda,q}(\cdot\restriction_{\{e\}} \mid \omega\restriction_{G-\{e\}})$. 

\subsubsection*{Size of largest component and drift functions}  

For $\lambda>1$, let $\theta_\lambda$ be the unique positive root of $e^{-\lambda x}=1-x$.
Recall the following tail estimates for $\sL_1$ in $\cG(n,p)$.

\begin{fact}[e.g., cf.~{\cite[p.~109]{JLR}}]\label{fact:cluster-exp-decay}
Consider $\cG(n,p)$ with $pn=\lambda<1$. Then for any  $x$,
\[\pi_{n,\lambda,1}(|\cC_{x}|\geq k ) \leq e^{-\frac {(1-\lambda)^2k}2}\,.
\]
In particular, $\pi_{n,\lambda,1}(\sL_1\geq k) \leq n \exp\big(-(1-\lambda)^2 k/2\big)$.
\end{fact}

\begin{proposition}[{\cite[Lemma 5.4]{LNNP14}}]\label{fact:giant-component}
Consider $\cG(n,p)$ with $np=\lambda>1$. There exists $c(\lambda)>0$ such that for every $\epsilon>0$,
\[\pi_{n,\lambda,1}( |\sL_1-\theta_\lambda|\geq \epsilon n) \lesssim e^{-c\epsilon^2 n}\,.
\]
\end{proposition}

For the proof of Theorem~\ref{mainthm:sw}, following~\cite{GoJe99,GSV15} define the drift function for the average size of the largest color class of the Swendsen--Wang dynamics 
\[
F_\lambda(z)=\left\{
\begin{array}{ll}
\theta_\lambda + \tfrac 1q (1-\theta_{\lambda z}) \qquad & \mbox{for $z> 1/\lambda$}\\ 
\tfrac 1q \qquad & \mbox{for $z\leq 1/\lambda$}
\end{array}
\right\}\,.
\]
The function $F_\lambda(z)$ has, for some values of $\lambda$ a second fixed point besides $\frac 1q$, which we denote by $a_\lambda>1/q$, which solves
\[\log \tfrac {(q-1)a}{1-a}= \lambda(a-\tfrac{1-a}{q-1})\,.
\]
\begin{proposition}[{\cite[Lemma 5]{GSV15}}] If $\lambda>\lambda_s$, the fixed point $a_\lambda$ is such that $\lambda a_\lambda >1$ and moreover if $b_\lambda=\frac{1-a_\lambda}{q-1}$, we have $\lambda b_\lambda<1$. Moreover, if $q>2$ and $\lambda>\lambda_s$, $a_\lambda$ is a Jacobian attractive fixed point of $F_\lambda(z)$ so that $|F'(a_\lambda)|<1$.
\end{proposition}

Similarly to the above, we can define the function $f$ given by 
\[
f(\theta)= \theta_{\lambda(1+(q-1)\theta)/q}\,,
\] which governs the mean drift of the size of the giant component in Chayes--Machta dynamics. We can also define $\Theta_r$ to be the largest solution to $e^{-\lambda x}= 1-\frac{qx}{1+(q-1)x}$.
Following  \cite{BlSi14}, let $\Theta_{\mathrm{min}}(\lambda,q)=\max\{0,(q-\lambda)/(\lambda(q-1))\}$, observe that if $\lambda<\lambda_S$, $\lambda(\Theta_{\mathrm{min}} +q^{-1}(1-\Theta_{\mathrm{min}}))=1$, and define the drift function $g(\theta)=f(\theta)-\theta$.

\begin{proposition}[{\cite[Lemma 2.14]{BlSi14}}]
When $q>2$ and $\lambda>\lambda_s$, the drift function $g$ has two roots, $\Theta^*<\Theta_r$ in $(\Theta_{\mathrm{min}},1]$; moreover, $g$ is strictly positive on $(\Theta^*,\Theta_r)$.
\end{proposition}

\subsubsection*{Mixing time and spectral gap}
In this section, we introduce the quantities of interest regarding the time for the Swendsen--Wang and Glauber dynamics to reach equilibrium. Consider a Markov chain with finite state space $\Omega$ and transition matrix $P$ reversible with respect to $\pi$. For two measures $\nu,\pi$, define their total variation distance by
\[\|\nu-\pi\|_\tv = \sup_{A\subset \Omega} |\nu(A)-\pi(A)|= \tfrac 12 \|\nu - \pi\|_{\ell^1}\,.
\]
Then the mixing time of $P$ is defined as 
\[\tmix=\inf \left\{t:\max_{X_0\in \Omega} \|P^t(X_0,\cdot)-\pi \|_\tv<1/(2e)\right\}\,.
\]

A related quantity that is sometimes easier to work with is the spectral gap of $P$; Since $P$ is reversible with respect to $\pi$, we can enumerate its spectrum from largest to smallest as $1=\lambda_1>\lambda_2>...$; then the spectral gap of $P$ is defined as $\gap=1-\lambda_2$. The following is a standard comparison between the inverse spectral gap and the mixing time of a Markov chain with transition matrix $P$ (see e.g.,~\cite{LPW09}):
\begin{align}\label{eq:tmix-gap}
\gap^{-1}-1 \leq \tmix \leq \log(2e/\pi_{\min}) \gap^{-1}\,.
\end{align}

\subsubsection*{Spectral gap comparisons}

The following comparison inequalities between the aforementioned Markov chains are due to Ullrich.

\begin{proposition}[\cite{Ul13}]\label{thm:Ullrich-comparison}
Let $q \geq 2$ be integer.
Let $\gap_{\rc}$ be the spectral gap of Glauber dynamics FK model on a graph $G=(V,E)$ and let $\gap_{\textsc {sw}}$
be the
spectral gap of Swendsen--Wang.
Then
\begin{align}
%\gap_{\textsc p} &\leq 2 q^2(qe^{2\beta})^{4\Delta} \gap_{\textsc {sw}}\,,\label{eq-ullrich1}\\
(1-p+p/q)\gap_{\rc} &\leq\gap_{\textsc {sw}}\leq 8 \gap_{\rc}\, |E| \log |E| \,.\label{eq-ullrich2}
\end{align}
\end{proposition}
The proof of~\eqref{eq-ullrich2} further extends to all real $q> 1$, whence
\begin{align}\gap_{\rc} &\lesssim\gap_{\textsc {cm}}\lesssim \gap_{\rc}\, |E|\log |E| \label{eq-ullrich2-realq}\,,
\end{align}
as was observed (and further generalized) by  Blanca and Sinclair~\cite[\S5]{BlSi14}, where $\gap_{\textsc{cm}}$ is the spectral gap of Chayes--Machta dynamics.

\section{Slow mixing of Swendsen--Wang dynamics}
Towards the proof of Theorem~\ref{mainthm:sw}, we first establish some preliminary estimates. For $\omega\in \Omega_{\rc}$, we will frequently be interested in bounding the following quantity:
\[S_M:=S_M(\omega)=\{x\in V:|\cC_x|>M\}\,.
\]
The bottlenecks in the proofs of Theorems~\ref{mainthm:sw}--\ref{mainthm:fk} both rely  on the following estimate.

\begin{lemma}\label{lem:large-cluster-vertices}
Consider $\omega\sim \cG(n,p)$ with $np=\lambda<1$ fixed. There exists $c(\lambda)>0$ such that for every $\rho>0$, there exists $M_0(\lambda,\rho)$ such that for every $M\geq M_0$,
\[\pi_{n,\lambda} \left(|S_M|\geq \rho n\right) \lesssim e^{-c\rho n}\,.
\]
\end{lemma}
\begin{proof}
Recall that by Fact~\ref{fact:cluster-exp-decay}, there exists $c_1(\lambda)>0$ such that $\pi_{n,\lambda}(|\cC_{x}| \geq k)\leq e^{-c_1 k}$ for all $k$. Moreover, conditioned on other clusters, the remaining graph is distributed as $\cG(m,\lambda)$ for $m\leq n$, so that for any $\ell$ vertices $y_1,...,y_\ell$,
\[ \pi_{n,\lambda}\bigg(|\cC_{x}|\geq k \;\Big|\; \cC_{y_1}\,,\ldots\,,\cC_{y_\ell}~,~\cC_x \cap \big(\mbox{$\bigcup_{i=1}^\ell \cC_{y_i}$}\big)=\emptyset\bigg) \leq e^{-c_1 k}\,.\]

Let $Y_M$ be the number of clusters with at least $M$ vertices; then  $Y_M \preceq \Bin(n,e^{-c_1M})$ so that by Azuma--Hoeffding inequality,
\[\pi_{n,\lambda}(Y_M \geq n e^{-c_1 M}+t ) \leq e^{-t^2/(2n)}\,.
\]
Now let 
\[ K=\left(e^{-c_1M}+M^{-1}\right)n \,;\] plugging into the Azuma-Hoeffding bound, we obtain
\begin{align}\label{eq:conditional-cluster-size} \pi_{n,\lambda}(|S_M|\geq \rho n) & \leq \pi_{n,\lambda}\left(\sum _{i=1}^{K} \sL_i\geq \rho n  \right)+ e^{-n/(2 M^2)} \,.
%& \leq \max_{k\leq K} \frac{\mathbb P(\sum_{i=1}^k |C_i| \geq \rho n)}{\mathbb P(\bigcap _{i=1}^k \{|C_i|\geq M\})} +e^{-cM^{-2} n}
\end{align}
In order to bound the right-hand side above, fix vertices $x_1,...,x_K$; the joint law of $\cC_{x_1},...,\cC_{x_K}$ is dominated by the sum of $K$ i.i.d.\ random variables $Z_1,...,Z_K$, where, for some $a(\lambda),b(\lambda),\nu(\lambda)>0$ (independent of $M$ and $n$), $Z_1$ is sub-exponential with parameters $(\nu,b)$ and has mean $a$. By the definition of $K$, for any sufficiently large $M$ (depending on $\rho$), $K\mathbb E[Z_i]=Ka \leq \rho n/2$. By a union bound and symmetry, we have
%for subexponential random variables,
\begin{align*}
 \pi_{n,\lambda}\bigg(\sum_{i=1}^K \sL_i\geq \rho n\bigg) & \leq \binom nK \pi_{n,\lambda} \bigg(\sum_{i=1}^K |\cC_{x_i}| \geq \rho n\bigg) \\
& \leq \left(\frac {en}K\right)^K \pi_{n,\lambda} \left(\sum_{i=1}^K Z_i \geq K\mathbb E[Z_i] +\rho n/2\right)\,.
\end{align*}
Moreover, $\sum_{i=1}^K Z_i$ is also sub-exponential with parameters $(K \nu,b)$. Therefore, there exists $c_2(\lambda)>0$ so that for all $\rho>0$, there exists $M_0(\lambda,\rho)$ such that for all $M\geq M_0$,
\begin{align*}
\pi_{n,\lambda}\bigg(\sum_{i=1}^K \sL_i\geq \rho n\bigg) \leq \left( \frac{e}{e^{-c_1M}+M^{-1}}\right)^{(e^{-c_1M}+M^{-1} )n} e^{-\frac {\rho n}{4b}} \lesssim e^{-c_2 \rho n}\,.
\end{align*}
Plugging this bound in to~\eqref{eq:conditional-cluster-size} concludes the proof.
\end{proof}

In the coloring stage of the Swendsen--Wang dynamics, the following simple application of a Chernoff-Hoeffding inequality proves useful.

\begin{lemma}\label{lem:hoeffding}
Consider an FK realization $\omega$ on $n$ vertices and suppose $|S_M(\omega)| \leq \epsilon n$ for some $M>0$. Independently color each cluster of $\omega$ collectively red with probability $\alpha\in [0,1]$, and let $R$ be the set of all red vertices. For all $\delta>0$, 
\[\mathbb P (\left| |R|-\alpha n\right| \geq (\epsilon+\delta) n) \leq 2\exp({-\tfrac{\delta^2 n}{2 M^2}})\,.
\]
\end{lemma}
\begin{proof}
We consider $\mathbb P(|R| \geq (\alpha+\epsilon +\delta)n)$ and $\mathbb P(|R|\leq (\alpha-\epsilon-\delta)n)$ separately. To bound the former, it suffices to prove an upper bound on 
\begin{align*}
\mathbb P(|R\cup S_M|\geq (\alpha+\epsilon+\delta)n) & =\mathbb P(|R-S_M|\geq (\alpha+\epsilon+\delta)n - |S_M|) \\
% & \leq \mathbb P(|R-S_M|\geq (p+\delta)n) \\
 & \leq \mathbb P(|R-S_M| \geq (\alpha+\delta)n)\,,
\end{align*}
which by Hoeffding's inequality satisfies 
\[\P (|R-S_M| - \alpha(n-|S_M|) \geq \delta n+\alpha|S_M|) \leq e^{-\frac {(\delta n+\alpha|S_M|)^2}{2(n-|S_M|)M^2}}\leq e^{- {\delta ^2 n}/({2M^2})}\,.
\]
Similarly bounding $\mathbb P(|R| \leq (\alpha-\epsilon-\delta)n) 
%& \leq \mathbb P(|R- S_M|\leq (r-\epsilon -\delta)n) \\ 
%%& \leq \mathbb P(|R-S_M| \leq (p-\delta) n -|S_M|) \\
%& 
\leq \mathbb P (|R-S_M| \leq (\alpha-\delta) n)
$
by Hoeffding's inequality and combining the two via a union bound concludes the proof.
\end{proof}

We prove Theorem~\ref{mainthm:sw} for $q>2$ separately for $\lambda$ that is below, above and at $\lambda_c$.

\subsection{The supercritical regime: proof of Theorem~\ref{mainthm:sw} for the case $\lambda\in (\lambda_c, \lambda_S)$}
To prove Theorem~\ref{mainthm:sw} for  $\lambda\in (\lambda_c,\lambda_S)$, let  $\rho>0$, and define the set of configurations,
\[A_\rho = \left\{\sigma \in \{1,...,q\}^n: \max_{r=1,..,q} \left | \sum_{i=1}^n \boldsymbol 1\{\sigma_i=r\} - \frac nq\right|<\rho n\right\}\,.
\]
Now consider the Markov chain $(X_t)_{t\geq 0}$ and let $v_t = (v_t^1,...,v_t^q)$ be the corresponding vector counting the number of sites in each state in $X_t$. We need the following claim. 

\begin{claim} \label{claim:want-to-show}
Consider Swendsen--Wang dynamics with $\lambda=n(1-e^{-\beta/n})$ for $\lambda<\lambda_S$; there exists $\rho_0(\lambda,q),c(\rho,\lambda,q),C(\lambda,q)>0$ such that that for every $\rho<\rho_0$
\begin{equation}\label{eq:want-to-show}
\max_{X_0 \in A_\rho} \mathbb P_{X_0} (X_1\notin A_\rho) \lesssim Ce^{-cn}\,.
\end{equation}
\end{claim}
\begin{proof}
Consider a fixed $X_0\in A_\rho$. In the $\cG(n,p)$ step of the Swendsen--Wang dynamics, we consider the color components separately. 
For each of the $q$ colored components a new edge configuration is sampled according to $\pi_{v_0^i, \lambda}$ where $i=1,...,q$; call the edge configuration we obtain $\omega_1^i$ and note that by definition of Swendsen--Wang dynamics, the clusters of $\{\omega_1^i\}_{i=1}^q$ will all be disconnected. Then since $\|v_0 - (\frac nq,...,\frac nq)\|_\infty <\rho n$ and $\lambda<\lambda_S=q$, if $\rho<1-\lambda/q=:\rho_0$, every colored component is sub-critical in the $\cG(n,p)$ step. Thus, for all $i=1,...,q$, by Lemma~\ref{lem:large-cluster-vertices}, for some $c(\lambda)>0$, if $\rho<1-\lambda/q$, for every $M\geq M_0(\lambda,\rho)$ and every $\delta>0$,
\[\mathbb P_{X_0} (|S_M(\omega_1^i)| \geq \delta n) = \pi_{v_0^i,\lambda}(|S_M|\geq \delta n) \lesssim e^{-c\delta n}\,.
\]
Union bounding over the $q$ different such components, we obtain 
\[\mathbb P_{X_0} \big( \mbox{$\bigcup_{i=1}^q$} \{|S_M (\omega_1^{i} )| \geq \delta n\}\big) \lesssim e^{-c\delta n}\,.
\]
In that case, if $\delta= \frac \rho{2q}$ and $\omega_1$ is the edge configuration induced on the whole graph after the $\cG(n,p)$ step of the dynamics, there exists $c(\lambda,q)>0$ so that for $M\geq M_0(\lambda,\rho)$,
\[\mathbb P_{X_0} \big (|S_M (\omega_1)| \geq \tfrac {\rho n}2\big ) \lesssim e^{-c\rho n}\,. \]
We can then split up
\[\mathbb P_{X_0} (X_1 \notin A_\rho) \leq \mathbb P_{X_0} (|S_M (\omega_1)| \geq \tfrac {\rho n}2)+ \P_{X_0} (X_1 \notin A_\rho\mid |S_M (\omega_1)| < \tfrac {\rho n}2)\,,
\] 
and consider the coloring step of the Swendsen--Wang dynamics. Then we obtain
\begin{align*}
\mathbb P_{X_0}(X_1 \notin A_{\rho}\mid |S_M (\omega_1)|  < \tfrac{\rho n}2) & \leq \mathbb P_{X_0} (\mbox{$\bigcup_{i=1}^q$} \{|v_1^i- \tfrac nq | \geq \rho n\}\mid |S_M (\omega_1)|< \tfrac {\rho n}2)\,.
\end{align*}
By an application of Lemma~\ref{lem:hoeffding} with $\epsilon=\delta=\frac {\rho}2$ and a union bound, the above is, for $\rho<1-\lambda/q$ and $M\geq M_0(\lambda,\rho)$, bounded above by
\begin{align*}
 2q \exp\left (-\rho^2 n/(8M^2)\right)\,. %\\ 
% & + q\mathbb P_{X_0} (v_1^1 \leq \tfrac nq- {\rho n} \mid X_1 \restriction_{S_{M}(\omega^1_1)} \neq 1,|S_M (\omega_1)|\leq \rho n/2)\,.
\end{align*}
%Observe that by linearity of expectation, we always have $\mathbb E [v_1^1] =\tfrac nq$ so that by a Chernoff bound, bounding the variance on $S_M(\omega_1)^c$ by $M^2 n$,
%\[\mathbb P_{X_0} (X_1 \notin A_\rho ) \leq 4q\exp\left({\tfrac{ -c\rho ^2 n}{4M^2}}\right)\,;
%\]
Since all the above estimates were uniform in $X_0\in A_\rho$, we obtain the desired.
\end{proof}

 By Corollary~\ref{cor:Potts-phase-transition}, since $\beta$ is such that $\lambda>\lambda_c$, for every small $\rho>0$, we have $\mu_{n,\lambda,q} (A_\rho^c)>\frac 12$. If $X_0$ is  such that $v_0 = (\frac nq,...,\frac nq)$, clearly $X_0\in A_\rho$, and by Claim~\ref{claim:want-to-show} and a union bound, since $\lambda<\lambda_S$, there exists $c(\rho,\lambda,q)>0$ such that for every $\rho<\rho_0$,
\[\mathbb P_{X_0} (X_{Ce^{cn/2}}\in A_\rho^c) \lesssim e^{-cn/2} \,.
\]
The definition of total variation mixing time then implies $\tmix \gtrsim e^{cn/2}$ as desired. 
\qed

\subsection{The subcritical regime: proof of Theorem~\ref{mainthm:sw} for $\lambda\in (\lambda_s, \lambda_c)$}
We first prove the following consequence of Lemma~\ref{lem:large-cluster-vertices}.

\begin{lemma}\label{lem:supercritical-large-cluster}
Consider $\cG (n,p)$ with $np=\lambda>1$. There exist $c(\lambda),c'(\lambda)>0$ such that for every $\rho>0$ and $\epsilon>0$ sufficiently small and for every $M\geq M_0(\lambda,\rho)$, we have 
\[
\pi_{n,\lambda} (\{|\sL_1-\theta_\lambda|\geq \epsilon n)\}\cup  \{|S_M-\cC_1|\geq \rho n\}) \lesssim e^{-c\rho n}+e^{-c'\epsilon^2 n}\,.
\] 
\end{lemma}
\begin{proof}By a union bound, rewrite the left-hand side above as 
\begin{align*}
\pi_{n,\lambda} (\{|\sL_1 -& \theta_\lambda|\geq \epsilon n\} \cup \{|S_M-\cC_1| \geq \rho n\}) \\ 
\leq & \,\pi_{n,\lambda}(|S_M-\cC_1| \geq \rho n\mid |\sL_1 -\theta_\lambda| < \epsilon n)  + \pi_{n,\lambda} (|\sL_1-\theta_\lambda|\geq \epsilon n)\,.
\end{align*}
Since $\lambda>1$, by Fact~\ref{fact:giant-component}, we have that $\pi_{n,\lambda}(\omega:|\sL_1 -\theta_\lambda|\geq \epsilon n)\leq e^{-c\epsilon^2 n}$ for some $c(\lambda)>0$. Now suppose $\cL_1\geq (\theta_\lambda-\epsilon) n$ and note that conditioning on $\cC_1$, since the remaining graph is disconnected from $\cC_1$, it must be distributed as $\cG(n-\sL_1,p)$ which, since $n-\sL_1 \leq (1-\theta_\lambda+\epsilon)n$,  is subcritical for small $\epsilon$. In that case, by Lemma~\ref{lem:large-cluster-vertices}, given $n-\sL_1\leq(1-\theta_\lambda+\epsilon) n$, there exists $c(\lambda)>0$ such that for every $\rho>0$, there exists $M_0(\lambda,\rho)>0$ so that for $M\geq M_0$,
\[\pi_{n-\sL_1,\lambda}(|S_M| \geq \rho n)\leq e^{-c \rho n}\,;
\]
combined  with the union bound, this implies the desired.
\end{proof}

The proof of Theorem~\ref{mainthm:sw} for $\lambda\in (\lambda_s,\lambda_c)$  is a slight modification of the proof for $\lambda\in(\lambda_c,\lambda_S)$. Recall the definitions of $\theta_\lambda,a_\lambda$ and $b_\lambda$ from~\S\ref{sec:prelims}. Fix $\lambda>\lambda_s$. In decreasing order, let the number of vertices in  each color class of $\sigma$ be $v^1,...,v^q$ and let
\[A_\rho=\left\{\sigma\in \{1,...,q\}^n:| v^1-a_\lambda| \leq \rho n, v^2 \leq \tfrac  {n-v^1}{q-1}+\rho n \right\}\,.
\]
By Corollary~\ref{cor:Potts-phase-transition}, since $\lambda<\lambda_c$, for sufficiently small $\rho$, we have $\mu_{n,\lambda,q} (A_\rho ^c) > \frac 12$. Therefore, it suffices by definition of total variation mixing to prove the following. 
\begin{claim}\label{claim:want-to-show-2}
Consider Swendsen--Wang dynamics with $\lambda=n(1-e^{-\beta/n})$ for $\lambda>\lambda_s$; there exist $\rho_0(\lambda,q),c(\rho,\lambda,q),C(\lambda,q)>0$ such that for every $\rho<\rho_0$,
\begin{align}\label{eq:want-to-show-2}
\max_{X_0 \in A_{\rho}} \mathbb P_{X_0} (X_1\notin A_\rho) \lesssim Ce^{-cn}\,;
\end{align}
\end{claim}
\begin{proof}
Fix any $X_0\in A_\rho$ and let $(v_0^1,...,v_0^q)$ be its corresponding color class vector. By definition of $a_\lambda$, for some $\rho'(\lambda,q)>0$ there exists $\gamma\in (F'(a_\lambda),1)$ such that if $|v_0^1-a_\lambda|\leq \rho'n$, we have $|F(v_0^1/n)-a_\lambda|<\gamma|v_0^1/n - a_\lambda|$. From now on we take $\rho<\rho'$.

Consider the $\cG(n,p)$ step of the Swendsen--Wang dynamics. Since $\lambda>\lambda_s$, $\lambda a_\lambda>1$ and $\lambda b_\lambda<1$, so that for $\rho>0$ sufficiently small, the first colored class of $X_0$ will be supercritical in the $\cG(n,p)$ step and the other $q-1$ will all be subcritical; call the $q$ random graph configurations we obtain in this step $\omega_1^i$ for $i=1,...,q$. Now fix such a $\rho>0$ and let $\epsilon=\frac {(1-\gamma)\rho} {2(q+1)}$. By Fact~\ref{fact:giant-component}, we obtain that for some $c(\lambda)>0$, 
\[\mathbb P_{X_0}(|\sL_1(\omega_1^1)- v_0^1 \theta_{\lambda v_0^1/n} | \geq \epsilon n) \lesssim e^{-c\epsilon^2 n}\,.
\]
Moreover, by Lemma~\ref{lem:large-cluster-vertices}, we also have for some $c(\lambda)>0$, for every $M\geq M_0(\lambda,\epsilon)$,
\[\mathbb P_{X_0} \big(|\sL_1(\omega_1^1)- v_0^1\theta_{\lambda v_0^1/n}| \geq \epsilon n\}\cup \mbox{$ \bigcup_{i=1}^q$} \{|S_M (\omega_1^{i} )-\cC_1(\omega_1^1)| \geq \epsilon n\big) \lesssim e^{-c\epsilon n}\,.
\]
On the complement of the above event, $\omega^1$ has a single giant component of size $\theta n$ for $\theta n\in (v_0^1\theta_{\lambda v_0^1/n} -\epsilon n, v_0^1\theta_{\lambda v_0^1/n }+\epsilon n)$,
 and $|S_M-\cC_1|\leq q\epsilon n$. By Lemma~\ref{lem:hoeffding}, with probability $1-e^{-c \theta n}$, the largest color class of $X_1$ will be the one containing $\cC_1(\omega_1^1)$ so without loss, we also assume that is the case.

At that stage, observe that  $\mathbb E[v_1^1 \mid \theta] = \theta n +\frac 1q (1-\theta)n $ and $\mathbb E[v_1^i\mid \theta] = \frac 1q (1-\theta) n $ for $i\neq 1$. Then, first assigning the giant component a color, then using Lemma~\ref{lem:hoeffding}, we obtain that for some $c(M,\lambda)>0$, for every $M\geq M_0(\lambda,\epsilon)$,
\[\mathbb P_{X_0} \left(|v_1^1-F(v_0^1/n)|\geq q\epsilon n+ \tfrac {q-1}q \epsilon n+\delta n \right)\lesssim e^{- c \delta^2 n}+e^{-c\epsilon n}\,.
\]
By a similar bound on the other $q-1$ coloring steps and the choice $\delta=(1-\gamma)\rho/2$,
\[\mathbb P_{X_0} \left( \| (v_1^1,...v_1^q) - \left(F(v^1_0/n), \tfrac {1-F(v_0^1/n)}{q-1},...,\tfrac{1-F(v_0^1/n)}{q-1}\right)\|_\infty \geq (1-\gamma) \rho n\right) \lesssim e^{-c\epsilon n}\,.
\]
By the choice of $\gamma$ and the triangle inequality, this implies 
\[\mathbb P_{X_0}(X_1\notin A_\rho) \leq \mathbb P_{X_0} \left (\|(v_1^1,...,v_1^q)- (a_\lambda,b_\lambda,...,b_\lambda)\|_\infty \geq \rho n\right) \lesssim e^{-c\epsilon n}\,,
\]
which by uniformity of the estimates over $X_0\in A_\rho$, concludes the proof.
\end{proof}

\subsection{The critical point: proof of Theorem~\ref{mainthm:sw} for $\lambda= \lambda_c$}
In Corollary~\ref{cor:Potts-phase-transition}, for every $q>2$, either $\gamma(q)\geq\frac12$ in which case Claim~\ref{claim:want-to-show-2} concludes the proof, or $1-\gamma(q)\geq \frac 12$ in which case Claim~\ref{claim:want-to-show} concludes the proof.
\qed

\section{Slow mixing of Glauber dynamics for the FK model}
Since for $q$ noninteger, Chayes--Machta dynamics activates a strict subset of the vertices at a time, we will need to use a modified argument to prove Theorem~\ref{mainthm:fk}. We instead construct a bottleneck set $S$ and bound its bottleneck ratio. For $A,B\subset \Omega$, let  
\[
Q(A,B)=\sum_{\omega\in A}\pi (\omega) P(\omega,B)=\sum_{\omega\in A}\pi(\omega)\sum_{\omega'\in B}P(\omega,\omega')\,,
\]
for a chain with stationary distribution $\pi$ and kernel $P$; the Cheeger constant of $\Omega$ is
\begin{align}\label{eq:conductance-gap}
\Phi=\max_{S\subset \Omega}\frac {Q(S,S^c)}{\pi(S)\pi(S^c)}\,,\qquad \mbox{and satisfies}\qquad 2\Phi\geq \gap \geq  {\Phi^2}/{2}\,.
\end{align}
In order to prove the lower bound of Theorem~\ref{mainthm:fk}, we prove such a lower bound on the inverse spectral gap of the Chayes--Machta dynamics, then using Proposition~\ref{thm:Ullrich-comparison} and a standard comparison between the spectral gap and mixing time~\eqref{eq:tmix-gap}, we obtain the desired for the Glauber dynamics. Before the proof of Theorem~\ref{mainthm:fk}, we prove some preliminary equilibrium bottleneck estimates for the mean-field FK model. 

The following lemma that was fundamental to the understanding of the distribution $\pi_{n,\lambda,q}$ in~\cite{BGJ} is very useful for the proof of Theorem~\ref{mainthm:fk}. 
\begin{lemma}[{\cite[Lemma 3.1]{BGJ}}]\label{lem:fundamental-lemma}
Fix $\alpha\in [0,1]$; consider a mean-field FK realization $\omega\sim \pi_{n,\lambda,q}$. Independently color each cluster of $\omega$ red with probability $\alpha$ and let $R$ be the collection of all red vertices. Conditional on $R$, the subgraph $\omega\restriction_{R}$ is distributed according to $\pi_{|R|,\lambda,rq}$ and the subgraph $\omega\restriction_{V-R}$ is distributed according to $\pi_{|V-R|,\lambda,(1-r)q}$.
\end{lemma}

The following corollary follows from iterating the process of Lemma~\ref{lem:fundamental-lemma} $\lfloor q\rfloor$ times. 

\begin{corollary}\label{cor:fundamental-lemma}
Consider a mean-field FK realization $\omega\sim \pi_{n,\lambda,q}$. Independently color each cluster of $\omega$ color $r_1,...,r_q$ with probability $q^{-1}$ each and $r_{0}$ otherwise. Then letting $R_0,R_1,...,R_q$ be the sets of vertices colored each of $r_0,...,r_q$, the subgraph restricted to $R_i$ for $i=1,...,q$ is distributed according to $\pi_{|R_i|,\lambda,1}$. The subgraph restricted to $R_0:= V-\bigcup_{i=1}^q R_i$ is distributed according to $\pi_{|R_0|,\lambda,q-\lfloor q\rfloor}$. Moreover, the distributions of the $\lceil q\rceil$ color classes are (conditionally on $R_0,...,R_q$) independent.
\end{corollary} 

(Note that when $q$ is integer, the set $R_0$ is deterministically empty.) Via Lemma~\ref{lem:fundamental-lemma}, we prove the following analogues of Lemmas~\ref{lem:large-cluster-vertices} and~\ref{lem:supercritical-large-cluster} when $q<1$. 

\begin{lemma}\label{lem:remainder-vertices}
Consider the mean-field FK model on $n$ vertices with parameters $(p,q)$ with $q<1$ and $np=\lambda<\lambda_c=q$. There exists $c(\lambda,q)>0$ such that for all $\rho>0$ sufficiently small, there exists $M_0(\lambda,\rho)>0$ such that for all $M\geq M_0$,
\[\pi_{n,\lambda,q}(|S_M|\geq \rho n) \lesssim e^{-c\rho n}\,.
\]
\end{lemma}

\begin{proof}
We prove the desired using Lemma~\ref{lem:fundamental-lemma}. Consider the random graph $\cG(m,p)$ with the choice of $m=\lceil q^{-1}n\rceil$; applying Lemma~\ref{lem:fundamental-lemma} to $\cG(m,p)$ with $\alpha=q$, by \cite[Lemma~9.1]{BGJ}, for all $\lambda\neq q$, we have $\mathbb P(|R|=n) \geq \frac C{\sqrt m}$, for some $C(\lambda)>0$.
Then, we can write for any event $A\subset \Omega_{\rc}$,
\begin{align}\label{eq:remainder-vertices}
\sum_{l=1}^m \mathbb P_{\mathrm{col},m,\lambda}(|R|=l) \pi_{m,\lambda}(\omega \restriction_{R}\in A\mid R, |R|=l)=\mathbb E_{\mathrm{col},m,\lambda} [\pi_{m,\lambda}(\omega \restriction_{R}\in A\mid R)]\,,
\end{align}
where $\mathbb P_{\mathrm{col},m,\lambda}$ is the distribution over colorings of $\omega$, averaged over realizations of $\omega\sim \pi_{m,\lambda}$.
Letting $A=A_{\rho,M}=\{|S_M|\geq \rho n\}$, for every $R$ the probability on the right-hand side is bounded above by $\pi_{m,\lambda}(A_{\rho,M})$ which, by Lemma~\ref{lem:large-cluster-vertices}, satisfies
\[\pi_{m,\lambda}(A_{\rho,M})\lesssim e^{-c\rho n}\,,
\]
for some $c(\lambda)>0$ and for every $\rho>0$ and every $M\geq M_0(\lambda,\rho)$.
But by Lemma~\ref{lem:fundamental-lemma}, 
\[\pi_{m,\lambda}(\omega \restriction_{R}\in \cdot \mid R,|R|=l) \stackrel{d} = \pi_{l,\lambda,q}(\omega\in \cdot)\,,
\]
which combined with $\mathbb P_{\mathrm{col},m,\lambda}(|R|=n)\geq C/\sqrt m$ implies 
\[\pi_{n,\lambda,q}(|S_M|\geq {\rho n}) \lesssim \sqrt {q^{-1}n} e^{-c\rho n}\,. \qedhere
\]
\end{proof}

\begin{lemma}\label{lem:supercritical-remainder-vertices}
Consider the mean-field FK model on $n$ vertices with parameters $(p,q)$ with $q<1$ and $np=\lambda>\lambda_c=q$. There exists $c(\lambda,q)>0$ such that for all $\rho>0$ sufficiently small, there exists $M_0(\lambda,\rho)>0$ such that for all $M\geq M_0$,
\[\pi_{n,\lambda,q}(|S_M-\cC_1|\geq \rho n) \lesssim e^{-c\rho n}\,.
\]
\end{lemma}
\begin{proof}
As before, consider $\cG(m,p)$ with $m=\lceil q^{-1} n\rceil$; by Lemma~\ref{lem:fundamental-lemma} with $\alpha=q$ and \cite[Lemma 9.1]{BGJ}, $\mathbb P(|R|=n)\geq C/\sqrt m$. Let $A=A_{\rho,M}= \{ |S_M-\cC_1|\geq \rho n\}$ in~\eqref{eq:remainder-vertices}. Then observe that $\pi_{m,\lambda}(\omega\restriction_R \in A_{\rho,M}) \leq \pi_{m,\lambda} (A_{\rho,M})$ and by Lemma~\ref{lem:supercritical-large-cluster}, $\pi_{m,\lambda} (A_{\rho,M}) \lesssim e^{-c\rho n}$. Altogether, plugging the above bounds in to~\eqref{eq:remainder-vertices} implies that there exists $c(\lambda)>0$ such that for all $\rho>0$ and all $M\geq M_0(\lambda,\rho)$,
\[\pi_{n,\lambda,q}(|S_M -\cC_1|\geq \rho n) \lesssim \sqrt {q^{-1} n} e^{-c\rho n}\,. \qedhere
\]
%Note first of all that with probability at least $1-e^{-c\rho n}$, if $\omega\sim \cG(m,p/n)$, we have $\cL_2 \leq \rho n/2$.
\end{proof}

\subsection{The supercritical/critical regime, $np=\lambda\in [\lambda_c,\lambda_S)$}
We first prove the desired mixing time lower bound for $\lambda\in [\lambda_c,\lambda_S)$, using the following bottleneck estimate.

\begin{lemma}\label{lem:bottleneck-estimate}
Consider the mean-field FK model on $n$ vertices with parameters $(p,q)$ where $q>2$ and $np=\lambda< \lambda_S$; there exists $c(\rho,M,\lambda,q)>0$ such that for all sufficiently small $\rho>0$, there exists $M_0(\lambda,\rho)$ such that for every $M\geq M_0$,
\[\pi_{n,\lambda,q} (\tfrac{\rho n}2 < |S_M|< \rho n \mid |S_M| < \rho n) \lesssim e^{-cn}\,.
\]
\end{lemma}
\begin{proof}
For $\rho,M>0$ define the sets
\begin{align*}
A_{\rho,M} & = \{\omega\in \Omega_{\rc}:|S_M(\omega)|<\rho n\}\,, \\
 B_{\rho,M}& =\{\omega\in \Omega_{\rc}:\tfrac {\rho n} 2< |S_M(\omega)|<\rho n\}\,.
 \end{align*}
In order to bound $\pi_{n,\lambda,q}(B_{\rho,M}\mid A_{\rho,M})$, use the coloring scheme described in Corollary~\ref{cor:fundamental-lemma}. Let $\mathcal P$ be the set of all possible partitions of $\{1,...,n\}$ into $\lceil q \rceil$ sets, i.e., the set of all possible colorings of FK configurations. Denote by $\mathbb P_{\mathrm {col}}$ the probability measure over colorings $(R_0,...,R_{\lfloor q\rfloor})$ averaged over $\pi_{n,\lambda,q}$, and $\mathbb P _{\mathrm{col}} (\cdot \mid \mathcal F)$ the probability measure over such colorings, averaged over $\pi_{n,\lambda,q}(\cdot \mid \mathcal F)$. For every $\mathbf R \in \mathcal P$, 
\[\pi_{\mathbf R} = \bigg\{
\begin{array}{ll}
\pi_{|R_i|,\lambda,1} & \mbox{ on $R_i$ for $i=1,...,\lfloor q\rfloor$} \\
\pi_{|R_0|,\lambda,q-\lfloor q\rfloor} & \mbox{ on $V-\bigcup_{i=1}^{\lfloor q\rfloor} R_i=:R_0$}
\end{array}\,.
\] 
Then we can write, by Corollary~\ref{cor:fundamental-lemma}, 
\begin{align*}
\pi_{n,\lambda,q} (& B_{\rho,M} \mid  A_{\rho,M})=   \sum_{\textbf R \in \mathcal P}\mathbb P_{\mathrm{col}}(\textbf R \mid A_{\rho,M}) \pi_{\textbf R} (B_{\rho, M} \mid A_{\rho,M})\,.
\end{align*}
By Lemma~\ref{lem:hoeffding}, since $A_{\rho,M}$ implies $|S_M|\leq \rho n$, for every $i=1,...,\lfloor q \rfloor$,
\[\mathbb P_{\mathrm{col}} \bigg(\left | |R_1|-\tfrac nq\right| <2\rho n \given A_{\rho,M} \bigg) \leq 2e^{-\rho^2 n/(2M^2)}\,.
\]
If $||R_i|-\frac nq| <2 \rho n$ for all $i=1,...,\lfloor q \rfloor$, we are left with a remainder set satisfying 
\[|R_0| \in \left ((1 - \tfrac {\lfloor q\rfloor}q- {2\rho \lfloor q\rfloor})n, (1 - \tfrac {\lfloor q \rfloor}q+ {2\rho \lfloor q \rfloor})n\right)\,.
\]
Define the event $\Gamma_\rho$ over colorings of the mean-field FK model as
\[\Gamma_\rho = \left \{\textbf R\in \mathcal P: \left| |R_i|-\tfrac nq \right|<2\rho n \mbox{ for all } i=1,...,\lfloor q \rfloor\right\}\,,
\]
so that the above conclusion can be written as 
\[
\mathbb P_{\mathrm{col}} (\Gamma_\rho^c \mid A_{\rho,M}) \lesssim \lfloor q\rfloor e^{-\rho^2 n/(2M^2)}\,.
\] 
Combined with the expression for $\pi_{n,\lambda,q}(B_{\rho,M}\mid A_{\rho,M})$, this implies that 
\begin{align*}
\pi_{n,\lambda,q} (B_{\rho,M} \mid A_{\rho,M}) & \leq  \max_{\textbf R \in \Gamma_\rho} \frac {\pi_{\textbf R} (B_{\rho,M})}{\pi_{\textbf R} (A_{\rho,M})} +\mathbb P_{\mathrm{col}} (\Gamma_{\rho}^c \mid A_{\rho,M}) \\ 
& \lesssim \max_{\textbf R \in \Gamma_\rho} \frac {\pi_{\textbf R} (|S_M|\geq \rho n/2)}{1-\pi_{\textbf R} (|S_M| \geq \rho n)}+e^{-\rho^2 n/(2M^2)}\,.
\end{align*}
By a union bound, the first term on the right-hand side is bounded above by 
\begin{align}\label{eq:conditional-probability-expansion}
\max_{\textbf R\in \Gamma_\rho} \frac {\pi_{|R_0|,\lambda,q-\lfloor q \rfloor} (|S_M|\geq \frac {\rho n}{2\lceil q \rceil})+\sum_{i=1,...,\lfloor q \rfloor} \pi_{|R_i|,\lambda,1} (|S_M| \geq \frac {\rho n}{2\lceil q \rceil})}{1-\pi_{|R_0|,\lambda,q-\lfloor q \rfloor} (|S_M|\geq \frac {\rho n}{\lceil q \rceil})-\sum_{i=1,...,\lfloor q \rfloor} \pi_{|R_i|,\lambda,1} (|S_M| \geq \frac {\rho n}{\lceil q \rceil})}\,.
\end{align}
We lower bound the numerator and upper bound the denominator simultaneously as they entail similar estimates.

Since $\lambda<\lambda_S=q$, there exists $\rho_0(\lambda,q)$ such that for all $\rho<\rho_0$, the random graph $\cG(\frac nq+2\lfloor q \rfloor \rho n,p)$ is subcritical and the FK model $\cG((1-\tfrac {\lfloor q \rfloor}q+ {2 \rho \lfloor q \rfloor})n,p,q-\lfloor q\rfloor)$ is also subcritical.
%Indeed since $p<q$, and $p_c(q-\lfloor q \rfloor)= q-\lfloor q \rfloor$, $p=p'q/(q-\lfloor q\rfloor)$ so that $p'<1-\lfloor q \rfloor$, as desired.
In other words, if $\rho<\rho_0(\lambda,q)$, for every $\textbf R\in \Gamma_\rho$, the distributions $\pi_{|R_i|,\lambda}$ for $i=1,...,\lfloor q \rfloor$ and $\pi_{|R_0|,\lambda,q-\lfloor q \rfloor}$ are all subcritical. As such, by Lemma~\ref{lem:large-cluster-vertices}, there exists $c(\lambda,q)>0$ such that for every $M\geq M_0(\lambda,\rho)$, 
\begin{align*}
\max_{\textbf R\in \Gamma_\rho}\sum_{i=1,...,\lfloor q\rfloor} \pi_{|R_i|,\lambda,1} (|S_M|\geq \tfrac{\rho n}{2\lceil q \rceil}) & \lesssim e^{-c\rho n/2}\,, \qquad \mbox{and} \\ 
\max_{\textbf R\in \Gamma_\rho} \sum_{i=1,...,\lfloor q\rfloor} \pi_{|R_i|,\lambda,1} (|S_M|\geq \tfrac{\rho n}{\lceil q \rceil}) & \lesssim e^{-c\rho n}\,.
\end{align*}
Similar bounds under $\pi_{|R_0|,\lambda,q-\lfloor q \rfloor}$ follow immediately for a different $c(\lambda,q)>0$ from Lemma~\ref{lem:remainder-vertices}. Altogether, this implies that for every $\rho<\rho_0$ and every $M\geq M_0(\lambda,\rho)$, there exists $c(\rho,M,\lambda,q)>0$ such that 
\[\pi_{n,\lambda,q} (B_{\rho,M}\mid A_{\rho,M}) \lesssim e^{-cn}\,. \qedhere
\]
\end{proof}

\begin{proof}[\emph{\textbf{Proof of Theorem~\ref{mainthm:fk}: the case $np=\lambda\in [\lambda_c,\lambda_S)$}}] For $\rho,M>0$, recall the definitions of $A_{\rho,M}$ and $B_{\rho,M}$.
By Proposition~\ref{prop:fk-phase-transition}, for $\lambda\in [\lambda_c,\lambda_S)$, for sufficiently small  $\rho>0$ and large $M$, there exists $c(\lambda,q)>0$ such that $\pi_{n,\lambda,q}(A_{\rho,M}^c)\geq c$. Then by~\eqref{eq:conductance-gap} it suffices  to prove an exponentially decaying upper bound on 
\begin{equation}\label{eq:want-to-show-fk-1}
\frac{Q(A_{\rho,M},A_{\rho,M}^c)}{\pi_{n,\lambda,q}(A_{\rho,M})}\lesssim  \max_{X_0\in A_{\rho,M}-B_{\rho,M}} P(X_0,A_{\rho,M}^c)+\pi_{n,\lambda,q}(B_{\rho,M} \mid A_{\rho,M})\,,
\end{equation}
where $P,Q$ are the transition matrix and edge measure, respectively, of the Chayes--Machta dynamics.
We first bound the first term in the right-hand side of~\eqref{eq:want-to-show-fk-1}. 

Consider some $X_0\in A_{\rho,M} - B_{\rho,M}$. In the activation stage of the Chayes--Machta dynamics, clusters are activated with probability $\frac 1q$; denote by $\mathcal A_1$ the set of activated vertices in this stage of the dynamics. Since $X_0\in A_{\rho,M}-B_{\rho,M}$, by Lemma~\ref{lem:hoeffding} with the choice of $\epsilon = \delta = \rho /2$,
\[
\mathbb P_{X_0} \left(\left||\mathcal A_1| -\tfrac nq\right| \geq \rho n\right)\leq 2e^{-\rho^2n/(8M^2)}\,.
\]
Since $\lambda<\lambda_S=q$, for $\rho<1-\lambda/q$, the random graph $\cG((\frac 1q+\rho) n,p)$ is subcritical. In that case, by Lemma~\ref{lem:large-cluster-vertices},  there exists $c(\lambda,\rho)>0$ such that for every $M\geq M_0(\lambda,\rho)$,
\begin{align*}
\mathbb P_{X_0} (X_1 \notin A_{\rho,M}\mid||\mathcal A_1| -\tfrac nq| < \rho n ) & \leq \mathbb P_{X_0} (|S_M (X_1\restriction_{\mathcal A_1})| \geq \tfrac {\rho n}2 \mid ||\mathcal A_1| -\tfrac nq| < \rho n)  \\
& \lesssim e^{-c\rho n/2}\,,
\end{align*}
Union bounding over the event $||\mathcal A_1|-n/q|\geq \rho n$ and its complement, there exists $c(\rho,M,\lambda,q)>0$ such that for every $\rho<1-\lambda/q$, for every $M\geq M_0(\lambda,\rho)$, 
\[\max_{X_0\in A_{\rho,M} - B_{\rho,M}} P(X_0,A_{\rho,M}^c) \lesssim e^{-cn}\,.
\]
Lemma~\ref{lem:bottleneck-estimate} yields a similar exponentially decaying upper bound on the second term on the right-hand side of~\eqref{eq:want-to-show-fk-1}, concluding the proof.
\end{proof}

\subsection{The subcritical/critical regime, $np=\lambda\in (\lambda_s,\lambda_c]$}

Recall the definitions of $\Theta^*(\lambda,q)$ and $\Theta_r(\lambda,q)$ corresponding to the drift function $g$. When $\lambda\in (\lambda_s,\lambda_c]$, we will need the following intermediate lemma, before proceeding to the analogue of Lemma~\ref{lem:bottleneck-estimate}. This is a straightforward adaptation of an argument of~\cite{BlSi14}. 

\begin{lemma}\label{lem:large-cluster-activated}
Consider the mean-field FK model on $n$ vertices with parameters $(p,q)$ with $np=\lambda\in(\lambda_s,\lambda_S)$; let $\omega_0\in A_{\rho,\epsilon,M} =\{\omega:\sL_1\geq (\Theta^*+\epsilon) n, |S_M-\cC_1|<\rho n\}$. Color $\cC_1$ red and independently color each cluster in $\omega_0-\cC_1$ red with probability $\frac 1q$; let $R$ be the set of all red vertices. Resample $\omega_0\restriction_{R}\sim\pi_{|R|,\lambda,1}$ and let $\omega_1$ be the resulting configuration on $n$ vertices; there exists $c(\rho,\epsilon,M,\lambda)>0$ so that for sufficiently small $\rho,\epsilon>0$, for every $M\geq M_0(\lambda,\rho)$, uniformly in $\omega_0\in A_{\rho,\epsilon,M}$,
\[ \mathbb P (\omega_1\notin \{\sL_1 \leq (\Theta^*+\epsilon) n\}) \lesssim e^{-cn}\,.
\]
\end{lemma} 
\begin{proof}
Fix any $\omega_0\in A_{\rho,\epsilon,M}$ and let $n\theta_0=\sL_1(\omega_0)$ for $\theta_0 \geq \Theta^*+\epsilon$. Then
\[\mathbb E\left[|R | \right] = \theta_0 n+\tfrac 1q (1-\theta_0) n=: \mu_0\,,
\]
so that by Lemma~\ref{lem:hoeffding}, for all $\rho>0$,
\[\mathbb P\left(||R| -\mu_0|\geq \rho n  \right)\leq 2e^{-\rho ^2 n/(8M^2)}\,.
\]
Therefore, we can write for every $\delta>0$,
 \begin{align*}
 \mathbb P(|\sL_1(\omega_1) - &\, nf(\theta_0)|\geq \delta n) \\
 & \leq \max_{a:|a-\mu_0|\leq \rho n} \pi_{a,\lambda} (|\sL_1 -nf(\theta_0)|\geq\delta n) + 2 e^{-\rho^2 n/(8M^2)}\,.
 \end{align*}
For all $\theta_0 \geq \Theta^*+\epsilon$, for sufficiently small $\rho>0$, using $\theta_0>\Theta^*>\Theta_{\mathrm{min}}$, since $\lambda<\lambda_S$,the random graph $\cG(\mu_0 -\rho n,p)$ is supercritical.
By continuity of $f$, for any $\delta>0$, there exists $\rho>0$ sufficiently small such that $\max_{a:|a-\mu_0|\leq \rho n}|f(\theta_0)-\theta_{\lambda a/n}|<\delta$; moreover, by Fact~\ref{fact:giant-component}, for every $\delta>0$
 \[\max_{a:|a-\mu_0|\leq \rho n} \pi_{a,\lambda}(|\sL_1 -\theta_{\lambda a/n}n |\geq \delta n)\lesssim e^{-c\delta^2 n}\,,
 \]
 for some $c(\lambda,\rho)>0$. Thus, for sufficiently small $\rho>0$, we have, for some $c(\rho,M,\lambda)>0$,
 \[\mathbb P(|\sL_1(\omega_1) -nf(\theta_0)| \geq 2\delta n) \lesssim e^{-c\delta^2 n}+ e^{-\rho^2 n/(8M^2)}\,.
 \] 
  
 It remains to argue that for $\epsilon>0$ sufficiently small, there exists $\delta>0$ such that for all $\theta_0\geq \Theta^*+\epsilon$, we have $nf(\theta_0) -2\delta n\geq (\Theta^* +\epsilon)n$. If $\theta_0> \Theta_r-\epsilon$, then by~\cite[Lemma 2.14]{BlSi14}, $f(\theta_0)\geq \Theta_r-\epsilon>\Theta^*+\epsilon$ and for small enough $\epsilon$ letting $\delta = \frac 12(\Theta_r -\Theta^* -2\epsilon)>0$ yields the desired. If $\theta_0 \leq\Theta_r-\epsilon$, since $g$ is positive on $(\Theta^*,\Theta_r)$, for $\epsilon$ small, $f(\theta_0)>\theta_0\geq \Theta^* +\epsilon$. By continuity of $f$, for $\epsilon<\frac 12 (\Theta_r-\Theta^*)$, letting $\delta= \frac 12 \min_{[\Theta^*+\epsilon,\Theta_r-\epsilon]} g$, we obtain
\[f(\theta_0)-2\delta \geq \theta_0 +g(\theta_0) -\min_{[\Theta^*+\epsilon,\Theta_r -\epsilon]} g\geq \theta_0 \geq \Theta^*+\epsilon\,.
\]
Together, for $\epsilon>0$ sufficiently small, there exists $c(\rho,\epsilon,M,\lambda)>0$ such that 
\[\mathbb P(\sL_1(\omega_1)\leq (\Theta^*+\epsilon)n ) \lesssim e^{-cn}\,.  \qedhere
\]
\end{proof}
The following is the analogue of Lemma~\ref{lem:bottleneck-estimate} in the presence of a giant component.
\begin{lemma}\label{lem:bottleneck-estimate-2}
Consider the mean-field FK model on $n$ vertices with parameters $(p,q)$ with $q>2$ and $np=\lambda\in (\lambda_s,\lambda_S)$; for every $\rho,\epsilon,M>0$ let 
\[E_{\rho,\epsilon,M}= \{\sL_1\geq (\Theta^*+\epsilon) n, \tfrac {\rho n}2<|S_M - \cC_1| <\rho n\}\,.
\]
There exists $c(\rho,M,\lambda,q)>0$ such that for sufficiently small $\rho,\epsilon>0$, for $M\geq M_0(\lambda,\rho)$,
\[\pi_{n,\lambda,q} \left(E_{\rho,\epsilon,M} \mid \sL_1\geq (\Theta^*+\epsilon)n, |S_M-\cC_1| < \rho n\right) \lesssim e^{-cn}\,.
\]
\end{lemma}
\begin{proof}
Fix $np=\lambda>\lambda_s$ and for $\rho,\epsilon,M>0$, define the sets
\begin{align*}
A_{\rho,\epsilon,M}& =  \{ \sL_1\geq (\Theta^*+\epsilon) n,  |S_M-\cC_1|<\rho n\}\,, \\
B_{\rho,M}&  =\{\tfrac {\rho n}2<|S_M-\cC_1|<\rho n\}\,.
\end{align*}
We prove the lemma similarly to Lemma~\ref{lem:bottleneck-estimate}, after treating the giant component separately. Using the coloring scheme of Corollary~\ref{cor:fundamental-lemma}, with $\mathbb P_{\mathrm{col}}$ and $\pi_{\textbf R}$ defined as before,
% we only color clusters $r_1$ with probability $\frac 1q$ and collect all the other vertices in $R_0$. Denote by $\pi_{R_0\times R_1}$ as 
%\[\pi_{R_0\times R_1}\bigg\{
%\begin{array}{ll}
%\pi_{|R_1|,p/n,1} & \mbox{ on $R_1$} \\
%\pi_{|R_0|,p/n,q-1} & \mbox{ on $V-R_1=R_0$}
%\end{array}
%\] 
%which corresponds to $\pi_{n,p/n,q}$ conditioned on $R_1$ and $R_0$ by Lemma~\ref{lem:fundamental-lemma}.  
by considering the color class to which $\cC_1$ belongs, and using symmetry, we obtain
\begin{align*}
\pi_{n,\lambda,q} ( E_{\rho,\epsilon,M} \mid  A_{\rho,\epsilon,M}) \!= & \,\, \!\tfrac q{\lfloor q \rfloor} \sum_{\textbf R\in \mathcal P}  \mathbb P_{\mathrm{col}}(\textbf R \mid \cC_1\subset R_1, A_{\rho,\epsilon,M}) \pi_{\textbf R} (E_{\rho, \epsilon,M} \mid \cC_1 \subset R_1,A_{\rho,\epsilon,M}) \\ 
& \!\!\!\!\!\!\!\!\!\!\!\!\!\!\!\!\!\!\! + \tfrac q{q-\lfloor q \rfloor} \sum_{\textbf R\in\mathcal P} \mathbb P_{\mathrm{col}} (\textbf R \mid \cC_1 \subset R_0, A_{\rho,\epsilon,M}) \pi_{\textbf R}(E_{\rho,\epsilon,M} \mid \cC_1\subset R_0,A_{\rho,\epsilon,M})\,.
\end{align*}
Call the two sums on the right hand side $\textbf I$ and $\textbf {II}$ respectively and consider them separately. 
Conditional on $A_{\rho,\epsilon,M}$ and $\cC_1\subset R_1$,  if $\mu_{\textbf I}=(\Theta^*+\epsilon) n+\tfrac 1q (1-\Theta^*-\epsilon) n$, 
\[\mathbb P_{\mathrm{col}} \left( |R_1|\geq \mu_{\textbf I}- 2\rho n) \mid \cC_1\subset R_1, A_{\rho,\epsilon,M}\right) \leq e^{-\rho^2 n/(2M^2)}\,,
\]
where we used Lemma~\ref{lem:hoeffding} with $\epsilon=\delta=\rho$.
Following the proof of Lemma~\ref{lem:bottleneck-estimate}, let 
\[\Gamma_{\rho}^{\textbf I} = \left\{ \textbf R: |R_1|\geq \mu_{\textbf I}-  2\rho n, \big| |R_i|- \tfrac {n-|R_1|}{q-1}\big|<2\rho n\mbox{ for all $i=2,...,\lfloor q \rfloor$} \right\}\,.
\]
By  Lemma~\ref{lem:hoeffding} and a union bound, $\mathbb P_{\mathrm{col}}((\Gamma_\rho^{\textbf I})^c\mid \cC_1\subset R_1, A_{\rho,\epsilon,M}) \leq 2\lceil q\rceil e^{-\rho^2n/(2M^2)}$. 
 
Using the fact that for every $\epsilon>0$, $E_{\rho,\epsilon,M}\subset B_{\rho,M}$, we can write 
\begin{align*}
\textbf I & \leq \mathbb P_{\mathrm{col}} ((\Gamma_\rho^{\textbf I})^c \mid \cC_1 \subset R_1,A_{\rho,\epsilon,M}) + \max_{\textbf R \in \Gamma_\rho^{\textbf I}} \pi_{\textbf R} (E_{\rho,\epsilon,M} \mid  A_{\rho,\epsilon,M},\cC_1 \subset R_1) \\ 
& \lesssim \tfrac {q\lceil q \rceil}{\lfloor q \rfloor} e^{-\rho^2 n/(2M^2)} + \tfrac{q}{\lfloor q \rfloor} \max_{\textbf  R \in \Gamma_\rho^{\textbf I}} \frac{\pi_{\textbf R} (B_{\rho,M}\mid \sL_1 \geq (\Theta^*+\epsilon)n,\cC_1 \subset R_1)}{\pi_{\textbf R} ( A_{\rho,\epsilon,M}\mid \sL_1 \geq (\Theta^*+\epsilon)n,\cC_1\subset R_1)}\,.
\end{align*}
If $\textbf R \in \Gamma_\rho^{\textbf I}$, for sufficiently small $\rho>0$, the definition of $\Theta^*$ and $\lambda>\lambda_s$ implies $\cG(|R_1|,p)$ is supercritical, and both $\cG(\frac {n-|R_1|}{q-1}+2\rho n,p)$ and $\cG(|R_0|,p,q-\lfloor q\rfloor)$ are subcritical.
By a union bound we can expand the numerator above as at most
\begin{align*}
&\max_{\textbf R\in \Gamma_\rho^{\textbf I}} \bigg (\pi_{|R_1|,\lambda}  (|S_M -\cC_1|\geq \tfrac {\rho n}{2\lceil q \rceil}\mid \sL_1 \geq (\Theta^*+\epsilon)n) + \mbox{$\sum_{i=2,...,\lfloor q \rfloor}$} \pi_{|R_i|,\lambda} (|S_M| \geq \tfrac {\rho n}{2\lceil q \rceil})\\
& \qquad \qquad+\pi_{|R_0|,\lambda,q-\lfloor q \rfloor} (|S_M|\geq \tfrac {\rho n}{2\lceil q \rceil})\bigg)+e^{-c\Theta^* n} \,,
\end{align*}
and analogously, the denominator as at least
\begin{align*}
& \min_{\textbf R \in \Gamma_{\rho}^{\textbf I}} \bigg (1-\pi_{|R_1|,\lambda} (|S_M-\cC_1|\geq \tfrac {\rho n}{\lceil q \rceil}\mid \sL_1 \geq (\Theta^*+\epsilon)n)-\pi_{|R_0|,\lambda,q-\lfloor q \rfloor} (|S_M|\geq \tfrac {\rho n}{\lceil q\rceil})  \\
& \qquad \qquad  - \mbox{$\sum_{i=2,...,\lfloor q \rfloor}$} \pi_{|R_i|,\lambda} (|S_M| \geq \tfrac {\rho n}{\lceil q \rceil})\bigg)-e^{-c\Theta^* n}\,.
\end{align*}
(In both of the above, we paid a cost of $e^{-c\Theta^* n}$ for the assumption $\sL_1(\omega)=\sL_1(\omega\restriction_{R_1})$.)
By  Lemma~\ref{lem:supercritical-large-cluster}, (for every $\sL_1\geq (\Theta^*+\epsilon)n$ and $\textbf R\in \Gamma_\rho^{\textbf I}$, $\cG(|R_1|-\sL_1,p)$ is subcritical) there exists $c(\lambda,q)>0$ such that for sufficiently small $\rho, \epsilon>0$ and every $M\geq M_0(\lambda,\rho)$,
\begin{align*}
\max_{\textbf R\in \Gamma_\rho^{\textbf I}} \pi_{|R_1|,\lambda,1} (|S_M-\cC_1|\geq \tfrac {\rho n}{\lceil q \rceil}\mid \sL_1 \geq (\Theta^*+\epsilon)n) & \lesssim e^{-c\rho n}\,, \qquad \mbox{and} \\
\max_{\textbf R\in \Gamma_\rho^{\textbf I}}  \pi_{|R_1|,\lambda,1} (|S_M-\cC_1| \geq \tfrac {\rho n}{\lceil q \rceil}\mid \sL_1 \geq (\Theta^*+\epsilon)n) & \lesssim e^{-c\rho n}\,.
\end{align*}

Moreover, as in the proof of Lemma~\ref{lem:bottleneck-estimate}, by Lemmas~\ref{lem:large-cluster-vertices} and~\ref{lem:remainder-vertices}, we also have that for $i=2,...,\lfloor q\rfloor$ that there exists $c(\lambda,q)>0$ such that for every  $M\geq M_0(\lambda,\rho)$,
\begin{align*}
\max_{\textbf R\in \Gamma_\rho^{\textbf I}} \pi_{|R_i|,\lambda,1} (|S_M|\geq \tfrac{\rho n}{\lceil q\rceil}) & \lesssim e^{-c\rho n}\,, \qquad \mbox{and} \\
 \max_{\textbf R\in \Gamma_\rho^{\textbf I}} \pi_{|R_0|,\lambda,q-\lfloor q\rfloor} (|S_M|\geq \tfrac{\rho n}{\lceil q \rceil} ) & \lesssim \sqrt n e^{-c\rho n}\,.
\end{align*}
Clearly, analogous bounds hold for all of the above when replacing $ \frac {\rho n}{\lceil q \rceil}$ with $\frac {\rho n}{2\lceil q \rceil}$.
Combining all of the above bounds and plugging them in to the right-hand side of 
\[\textbf I \lesssim 
\max_{\textbf  R \in \Gamma_\rho^{\textbf I}} \frac{\pi_{\textbf R} (B_{\rho,M}\mid \sL_1 \geq (\Theta^*+\epsilon)n,\cC_1 \subset R_1)}{\pi_{\textbf R} ( A_{\rho,\epsilon,M}\mid \sL_1 \geq (\Theta^*+\epsilon)n,\cC_1\subset R_1)}+ e^{-\rho^2 n/(2M^2)}\,,
\]
yields an exponentially decaying upper bound on the sum $\textbf I$. The bound on the sum $\textbf {II}$ is very similar. Letting $\mu_{\textbf {II}}= (\Theta^*+\epsilon)n+\frac {q-\lfloor q \rfloor}q (1-\Theta^*-\epsilon) n$, we define
\[\Gamma_{\rho}^{\textbf {II}} = \left \{ |R_0|\geq \mu_{\textbf {II}} -2\rho n, \big||R_i|-\tfrac {n-|R_0|}{\lfloor q \rfloor}\big|<2\rho n\mbox{ for all $i=1,...,\lfloor q \rfloor$} \right\}\,.
\]
As before, by Lemma~\ref{lem:hoeffding}, we can write 
\[\textbf {II} \lesssim e^{-\rho^2 n/(2M^2)} + \max_{\textbf R \in \Gamma_\rho^{\textbf {II}}} \frac {\pi_{\textbf R}(B_{\rho,M} \mid \sL_1 \geq (\Theta^* +\epsilon)n,\cC_1\subset R_0)}{\pi_{\textbf R} (A_{\rho,\epsilon,M}\mid \sL_1 \geq (\Theta^*+\epsilon)n,\cC_1\subset R_0)}\,,
\]
and observe that for every $\textbf R\in \Gamma_\rho^{\textbf {II}}$, since $\lambda\in (\lambda_s,\lambda_S)$, for sufficiently small $\rho>0$, the FK model $\pi_{|R_0|,\lambda,q-\lfloor q \rfloor}$ is supercritical and the random graphs $\cG(|R_i|,\lambda)$  are  subcritical for all $i=1,...,\lfloor q \rfloor$. By Lemmas~\ref{lem:large-cluster-vertices} and~\ref{lem:supercritical-remainder-vertices}, there exists $c(\lambda,q)>0$ such that for every $\rho>0$ sufficiently small and every $M\geq M_0(\lambda,\rho)$, 
\begin{align*}
\max_{\textbf R\in \Gamma_\rho^{\textbf {II}}} \pi_{|R_0|, \lambda,q-\lfloor q \rfloor} &(|S_M-\cC_1| \geq \tfrac {\rho n}{\lceil q \rceil}\mid \sL_1 \geq (\Theta^*+\epsilon)n)  \\ 
& \leq \max_{\textbf R\in \Gamma_\rho^{\textbf {II}}} \pi_{|R_0|, \lambda,q-\lfloor q \rfloor}  (|S_M-\cC_1| \geq \tfrac {\rho n}{\lceil q \rceil}) \lesssim e^{-c\rho n}\,, \quad \mbox{and} \\
 \max_{\textbf R\in \Gamma_\rho^{\textbf {II}}} \pi_{|R_i|,\lambda,1} & (|S_M| \geq \tfrac{\rho n}{\lceil q\rceil}) \lesssim e^{-c\rho n} \quad \mbox{ for all $i=1,...,\lfloor q\rfloor$}\,,
\end{align*}
and by the same reasoning, analogous bounds hold when replacing $\frac {\rho n}{\lceil q \rceil}$ with $\frac {\rho n}{2\lceil q \rceil}$. Then expanding the fraction in the upper bound on $\textbf {II}$ as done in the bound on $\textbf I$ implies there exists $c(\lambda,q)>0$ such that for sufficiently small $\rho,\epsilon>0$ and every $M\geq M_0(\lambda,\rho)$,
\[\pi(E_{\rho,\epsilon,M}\mid A_{\rho,\epsilon,M}) \lesssim \textbf I+\textbf {II} \lesssim e^{-c\rho n}+e^{-c\Theta^* n}+e^{-\rho^2 n/(2M^2)}\,. \qedhere
\] 
\end{proof}

We are now in position to complete the proof of Theorem~\ref{mainthm:fk}.

\begin{proof}[\textbf{\emph{Proof of Theorem~\ref{mainthm:fk}: the case $np=\lambda\in (\lambda_s,\lambda_c]$}}]
The proof when $\lambda\in (\lambda_s,\lambda_c]$ is similar to the extension of slow mixing for the Swendsen--Wang dynamics when $\lambda\in [\lambda_c,\lambda_S)$ to $\lambda\in (\lambda_s,\lambda_c]$. 
Recall that for fixed $\lambda>\lambda_s$, the two zeros of $g(\theta)=f(\theta)-\theta$ were denoted  $\Theta^*<\Theta_r$ so that $g$ is positive on $(\Theta^*,\Theta_r)$. 
We again use a conductance estimate to lower bound the inverse gap of the Chayes--Machta dynamics. Define for every $\rho,\epsilon,M>0$,
\begin{align*}
A_{\rho,\epsilon,M}& =  \{ \sL_1\geq (\Theta^*+\epsilon) n,  |S_M-\cC_1|<\rho n\}\,, \\
E_{\rho,\epsilon,M}&  =\{\sL_1\geq (\Theta^*+\epsilon) n,  \tfrac {\rho n}2<|S_M-\cC_1|<\rho n\}\,.
\end{align*}
As in~\eqref{eq:want-to-show-fk-1}, by~\eqref{eq:conductance-gap} it suffices to show an exponentially decaying upper bound on 
\begin{align*}
\frac{Q(A_{\rho,\epsilon,M}, A_{\rho,\epsilon,M}^c)}{\pi_{n,\lambda,q}(A_{\rho,\epsilon,M})} \lesssim \max_{X_0\in A_{\rho,\epsilon,M}- E_{\rho,\epsilon,M}} P(X_0,A_{\rho,\epsilon,M}^c) + \pi_{n,\lambda,q}(E_{\rho,\epsilon,M}\mid A_{\rho,\epsilon,M})\,,
\end{align*}
for sufficiently small $\rho,\epsilon>0$ and large $M$; 
this is because by Proposition~\ref{prop:fk-phase-transition}, for all small enough $\epsilon,\rho$,  we have $\pi_{n,\lambda,q}(A_{\rho,\epsilon,M}^c) \geq c>0$. We bound the two terms above separately as in the proof for $\lambda\in [\lambda_c,\lambda_S)$. First of all, note by Lemma~\ref{lem:bottleneck-estimate-2} that the second term on the right-hand side is bounded above by $e^{-cn}$ for some $c(\rho,M,\lambda,q)>0$ for every sufficiently small $\epsilon,\rho>0$ and every $M\geq M_0(\lambda,\rho)$.

Now consider any $X_0\in A_{\rho,\epsilon,M}-E_{\rho,\epsilon,M}$ and bound $P(X_0,A^c_{\rho,\epsilon,M})$ under the Chayes--Machta dynamics. We split the transition probability of the Chayes--Machta dynamics into the case when $\cC_1(X_0)$ is activated and $\cC_1(X_0)$ is not activated; let $\mathcal A_1$ denote the set of activated vertices. If $\cC_1(X_0)\not \subset \mathcal A_1$, we have $\mathbb E[|\mathcal A_1| \mid \cC_1 \not \subset \mathcal A_1]\leq  \frac 1q (1-\Theta^*-\epsilon)n$ and since $X_0\in A_{\rho,\epsilon,M}$, by Lemma~\ref{lem:hoeffding}, if $\epsilon>\rho$, then
\[\mathbb P_{X_0}\big(|\mathcal A_1| \geq \tfrac 1q (1-\Theta^*-\epsilon)n+\epsilon n\mid \cC_1(X_0)\not \subset \mathcal A_1\big) \leq 2 e^{-\epsilon^2 n/(2M^2)}\,.
\]
If $|\mathcal A_1| \leq \frac 1q (1-\Theta^*-\epsilon) n+\epsilon n$, for sufficiently small $\epsilon>0$, since $\lambda<\lambda_S=q$, the random graph  $\cG(|\mathcal A_1|,p)$ is subcritical, in which case with probability at least $1-e^{-c\Theta^* n}$, $\cC_1(X_1)=\cC_1(X_0)$. By Lemma~\ref{lem:large-cluster-vertices}, there exists $c(\rho,M,\lambda,q)>0$ such that for $0<\rho<\epsilon$ sufficiently small and every $M\geq M_0(\lambda,\rho)$, 
\begin{align*}
\mathbb P_{X_0}(|S_M(X_1)-\cC_1&(X_1)| \geq \rho n\mid \cC_1(X_0)\not \subset \mathcal A_1) \\
%&  \leq \mathbb P(|S_M(X_1\restriction_{\mathcal A_1})| \geq \tfrac {\rho n}2 \mid \cC_1(X_0)\not \subset \mathcal A_1) \\
& \lesssim   \pi_{\frac 1q (1-\Theta^*+(q-1)\epsilon)n,\lambda,1} (|S_M|\geq \tfrac{\rho n}2) + e^{-\epsilon^2 n/(2M^2)} + e^{-c\Theta^* n} \\
& \lesssim e^{-c\rho n}+ e^{-\epsilon^2 n/(2M^2)}+e^{-c\Theta^* n}\,.
\end{align*}
Thus, for some $c(\rho,\epsilon,M,\lambda,q)>0$, for small enough $0<\rho<\epsilon$, and every $M\geq M_0(\lambda,\rho)$,
\[\max_{X_0\in A_{\rho,\epsilon,M}-E_{\rho,\epsilon,M}} \mathbb P_{X_0} (X_1\not \in A_{\rho,\epsilon,M} \mid \cC_1(X_0)\not \subset \mathcal A_1) \lesssim e^{-c n}\,.
\]

Now suppose that $\cC_1(X_0) \subset \mathcal A_1$; then one step of Chayes--Machta dynamics is described precisely by the set up of Lemma~\ref{lem:large-cluster-activated}, with $\rho$ replaced by $\rho/2$, yielding 
\[\max_{X_0\in A_{\rho,\epsilon,M}-E_{\rho,\epsilon,M}} \mathbb P_{X_0} (\sL_1 \leq (\Theta^*+\epsilon)n \mid  \cC_1(X_0) \subset \mathcal A_1) \lesssim e^{-c' n}
\] 
for some $c'(\epsilon,\rho,M,\lambda,q)>0$ for all sufficiently small $\epsilon,\rho>0$ and $M\geq M_0(\lambda,\rho)$. On the complement of that event, deterministically $\cC_1 (X_1)= \cC_1(X_1\restriction_{\cA_1})$.
By Lemma~\ref{lem:supercritical-large-cluster}, for some  $c(\lambda,q)>0$, for small $\epsilon,\rho>0$ and large $M\geq M_0(\lambda,\rho)$,
 \[\mathbb P(|S_M(X_1\restriction_{\mathcal A_1}) - \cC_1(X_1\restriction_{\cA_1})| \geq \rho n/2\mid \cC_1(X_0)\subset \mathcal A_1) \lesssim e^{-c\rho n/2}\,.
 \]
Combining the above, we deduce that there exists $c(\rho,\epsilon,M,\lambda,q)>0$ such that for all sufficiently small $0<\rho<\epsilon$, for every $M\geq M_0(\lambda,\rho)$,  we have $P(X_0,A_{\rho,\epsilon,M}^c)\lesssim e^{-cn}$, concluding the proof of Theorem~\ref{mainthm:fk} when $\lambda\in (\lambda_s,\lambda_c]$.
\end{proof}

\subsection*{Acknowledgment} R.G.\ thanks the theory group of Microsoft Research Redmond for its hospitality during the time some of this work was carried out. E.L.\ was supported in part by NSF grant DMS-1513403.

\bibliographystyle{abbrv}
\bibliography{mean-field-exponential-mixing}

\end{document}